\documentclass[final]{siamltex}

\usepackage{amssymb}
\usepackage{amsmath}
\usepackage[T1]{fontenc}
\usepackage[latin1]{inputenc}
\usepackage{amsfonts}
\usepackage{graphicx}
\usepackage{epsfig}
\usepackage{hyperref}

\newtheorem{remark}{Remark}

\title{Local Exponential $\mathbf{H^2}$ Stabilization of a $\mathbf{2\times 2}$ Quasilinear Hyperbolic System using Backstepping
\thanks{This work is an extension of the paper of the same title presented at the 2011 CDC-ECC conference.}
}

\author{
Jean-Michel Coron\thanks{Laboratoire Jacques-Louis Lions, Univ. Pierre et Marie Curie and Institut Universitaire de France, B.C. 187, 4 place Jussieu, 75252 Paris Cedex 05, France ({\tt coron@ann.jussieu.fr}). JMC was partially
supported by the ERC advanced grant 266907 (CPDENL) of the 7th Research Framework Programme (FP7).}
\and
Rafael Vazquez\thanks{Department of Aerospace Engineering, Universidad de Sevilla, Camino de los Descubrimiento s.n., 41092 Sevilla, Spain ({\tt rvazquez1@us.es}).}
 \and Miroslav Krstic\thanks{Department of Mechanical and Aerospace Engineering, University of California San Diego, La Jolla, CA 92093-0411, USA ({\tt  krstic@ucsd.edu}).}
\and Georges Bastin\thanks{Department of Mathematical Engineering, Univ. Catholique de Louvain, 4 Avenue G. Lema\^itre, 1348 Louvain-la-Neuve Belgium ({\tt Georges.Bastin@uclouvain.be}). Research supported by the Belgian Programme on Interuniversity Attraction Poles (IAP V/22)}
}

\begin{document}

\maketitle

\begin{abstract}
In this work, we consider the problem of boundary stabilization for a quasilinear $2\times2$ system of first-order hyperbolic PDEs. We design a new full-state feedback control law, with actuation on only one end of the domain,  which achieves $H^2$ exponential stability of the closed-loop system. Our proof uses a backstepping transformation to find new variables for which a strict Lyapunov function can be constructed. The kernels of the transformation are found to verify a Goursat-type $4\times4$ system of first-order hyperbolic PDEs, whose well-posedness is shown using the method of characteristics and successive approximations. Once the kernels are computed, the stabilizing feedback law can be explicitly constructed from them.

\end{abstract}



\pagestyle{myheadings}
\thispagestyle{plain}
\markboth{J.-M. CORON, R. VAZQUEZ, M. KRSTIC AND G. BASTIN}{LOCAL $H^2$ STABILIZATION OF A QUASILINEAR HYPERBOLIC SYSTEM}

\section{Introduction}
In this paper we are concerned with the problem of boundary stabilization for a $2\times2$ system of first-order hyperbolic \emph{quasilinear} PDEs, with actuation at only one of the boundaries. The quasilinear case is of interest since many relevant physical systems are described by $2\times2$ systems of first-order hyperbolic quasilinear PDEs, such as open channels~\cite{dos-santos, GugatLeugering, 2004-Gugat-Leugering-Schmidt,Halleux}, transmission lines~\cite{curro}, gas flow pipelines~\cite{gugat-2} or road traffic models~\cite{goatin}.

This problem has been considered in the past for  $2\times2$ systems~\cite{li1} and even $n\times n$ systems~\cite{li2}, using the explicit evolution of the Riemann invariants along the characteristics. More recently, an approach using control Lyapunov functions has been developed, for $2\times2$ systems~\cite{coron} and $n\times n$ systems~\cite{coron2}. These results use only static output feedback (the output being the value of the state on the boundaries). However, they do not deal with the same class of systems considered in this work (which includes an extra term in the equations); with this term, it has been shown in~\cite{bastin} that there are examples (even for linear $2\times2$ system) for which there are no control Lyapunov functions of the ``diagonal'' form $\int_0^1z(x,t)^TQ(x)z(x,t) dx$ (see  next section for notation) which would allow the computation of a static output feedback law to stabilize the system, even if feedback is allowed on both sides of the boundary.

Several other authors have also studied this problem. For instance, the linear case has been analyzed in~\cite{xu} (using a Lyapunov approach) and in~\cite{litrico} (using a spectral approach). The nonlinear case has been considered by~\cite{dick} and~\cite{gugat} using a Lyapunov approach, and in~\cite{prieur},~\cite{prieur2}, and~\cite{dos-santos} using a Riemann invariants approach.

The basis of our design is the backstepping method~\cite{krstic}; initially developed for parabolic equations, it has been used for first-order hyperbolic equations~\cite{krstic3}, delay systems~\cite{krstic5}, second-order hyperbolic equations~\cite{krstic2}, fluid flows~\cite{vazquez}, nonlinear PDEs~\cite{vazquez2} and even PDE adaptive designs~\cite{krstic4}. The method allows us to design a full-state feedback law (with actuation on only one end of the domain) making the closed-loop system locally exponentially stable in the $H^2$ sense. The gains of the feedback law are the solution of a  4 x 4 system of first-order hyperbolic linear PDEs, whose well-posedness is shown. The proof of stability is based on~\cite{coron2}; we construct a strict Lyapunov function, locally equivalent to the $H^2$ norm, and written in coordinates defined by the (invertible) backstepping transformation.

The paper is organized as follows. In Section~\ref{sect-prob} we formulate the problem. In Section~\ref{sect-linear} we consider the linear case and formulate a backstepping design that globally stabilizes the system in the $L^2$ sense. In Section~\ref{sect-main} we present our main result, which shows that the linear design locally stabilizes the nonlinear system in the $H^2$ sense. The proof of this result is given in Section~\ref{sect-proof}.
We finish in Section~\ref{sect:conclusions} with some concluding remarks. We also include an appendix with the proof of well-posedness of the kernel equations, and some technical lemmas.

\section{Problem Statement}\label{sect-prob}

Consider the system
\begin{equation}
z_t+\Lambda(z,x)z_x+f(z,x)=0,\quad x\in[0,1],\, t\in[0,+\infty),\label{eqn-z}
\end{equation}
where $z\,:\,[0,1]\times[0,\infty)\rightarrow \mathbb{R}^2$, $\Lambda\,:\,\mathbb{R}^2 \times [0,1] \rightarrow \mathcal{M}_{2,2}(\mathbb{R})$, $f\,:\,\mathbb{R}^2 \times [0,1] \rightarrow \mathbb{R}^2$, with $ \mathcal{M}_{2,2}(\mathbb{R})$ denoting the set of $2\times2$ real matrices. We assume that $\Lambda(z,x)$ is twice continuously differentiable with respect to $z$ and $x$, and we assume that (possibly after an appropiate state transformation) $\Lambda(0,x)$ is a diagonal matrix with nonzero eigenvalues $\Lambda_1(x)$ and $\Lambda_2(x)$ which are, respectively, positive and negative, i.e.,
\begin{equation}
\Lambda(0,x)=\mathrm{diag}(\Lambda_1(x),\Lambda_2(x)),\quad\Lambda_1(x)>0,\Lambda_2(x)<0,\label{eqn-speeds}
\end{equation}
where $\mathrm{diag}(\Lambda_1,\Lambda_2)$ denotes the diagonal matrix with $\Lambda_1$ in the first position of the diagonal and $\Lambda_2$ in the second.

We also assume that $f(0,x)=0$, implying that there is an equilibrium at the origin, and that $f$ is twice continuously differentiable with respect to $z$. Denote
\begin{equation}
\frac{\partial f}{\partial z}(0,x)=\left[
\begin{array}{cc}
f_{11}(x) & f_{12}(x) \\
f_{21}(x) & f_{22}(x)
\end{array}
\right],
\end{equation}
and assume that $f_{ij}\in\mathcal C^1\left([0,1]\right)$.

Denoting $z=[ z_1 \ z_2]^T$, we study classical solutions of the system under the following boundary conditions
\begin{equation}
z_1(0,t) =G_0 \left(
z_2(0,t)
 \right),
 \quad
z_2(1,t)=U(t),\quad\label{eqn-bcz12} t\in[0,+\infty)
\end{equation}
which are consistent (see~\cite{russell}) with the signs of (\ref{eqn-speeds}), at least for small values of $z$. We assume that $G_0(x)$ is twice differentiable and vanishes at the origin. In (\ref{eqn-bcz12}),
 $U(t)$ is the actuation variable, and our task is to find a feedback law for $U(t)$ to make the origin of system (\ref{eqn-z}),(\ref{eqn-bcz12}) locally exponentially stable.
\begin{remark}\em
The case with $f=0$ in (\ref{eqn-z}) was addressed in~\cite{coron} and~\cite{coron2} by using control Lyapunov functions to design a static output feedback law; this approach has been shown to fail in~\cite{bastin} for some cases with $f\neq0$, at least for a ``diagonal'' Lyapunov function of the form $\int_0^1z^T(x,t)Q(x)z(x,t) dx$.
\end{remark}
\section{Stabilization of $2\times2$ hyperbolic linear systems}\label{sect-linear}
Next, we present a new design, based on the backstepping method, to stabilize a $2\times2$ hyperbolic linear system; this procedure will be used later to \emph{locally} stabilize system (\ref{eqn-z}), (\ref{eqn-bcz12}).

Consider the system
\begin{equation}
w_t=\Sigma(x)w_x+C(x)w, \label{eqn-wlinear}\quad x\in[0,1],\, t\in[0,+\infty),
\end{equation}
where $w\,:\,[0,1]\times[0,\infty)\rightarrow \mathbb{R}^2$, $\Sigma,C\,: [0,1] \rightarrow \mathcal{M}_{2,2}(\mathbb{R})$,
where the matrices $\Sigma$ and $C$ are respectively diagonal and antidiagonal, as follows:
\begin{eqnarray}
\Sigma(x)&=&
\left(
\begin{array}{cc}
-\epsilon_1(x)  & 0     \\
0  & \epsilon_2(x)
\end{array}
\right),
C(x)=
\left(
\begin{array}{cc}
0  & c_1(x)     \\
c_2(x)  & 0
\end{array}
\right),\quad
\end{eqnarray}
where $c_1(x),c_2(x)$ are $\mathcal C([0,1])$ and $\epsilon_1(x),\epsilon_2(x)$ are $\mathcal C^1([0,1])$ functions, verifying that $\epsilon_1(x),\epsilon_2(x)>0$, and with boundary conditions
\begin{equation}\label{eqn-bculinear}
u(0,t)=qv(0,t),\quad
v(1,t)= U(t),
\end{equation}
where $q\in \mathbb{R}$ and the components of $w$ are $w=\left[
 u \ v
\right] ^T$.
Our objective is to design a full-state feedback control law for $U(t)$ to ensure that the closed-loop system is globally asymptotically stable in the $L^2$ norm, which is defined as $\Vert w(\cdot,t) \Vert_{L^2} = \sqrt{\int_0^1 \left(u^2(\xi,t)+v^2(\xi,t)\right) d\xi}$. There are two cases, depending on whether $q$ in (\ref{eqn-bculinear}) is nonzero or $q=0$. We first analyze the first case, thus assuming $q\neq 0$.

\subsection{Target system and backstepping transformation}
Our approach to designing $U(t)$, following the backstepping method, is to seek a mapping that transforms $w$ into a {\em target} variable $\gamma$ with asymptotically stable dynamics as follows:
\begin{equation}
\gamma_t=\Sigma(x)\gamma_x, \label{eqn-gammalinear}
\end{equation}
with boundary conditions
\begin{eqnarray}
\alpha(0,t)&=&q\beta(0,t),\quad
\beta(1,t)=0,\label{eqn-bc1beta}
\end{eqnarray}
 where the components of $\gamma$ are denoted as
\begin{eqnarray}
\gamma (x,t)=
\left[
\alpha(x,t) \ \
\beta(x,t)
\right]^T.
\end{eqnarray}
System (\ref{eqn-gammalinear}), (\ref{eqn-bc1beta}) verifies the properties expressed in the following proposition.
\begin{proposition}\label{pr-target}
Consider system (\ref{eqn-gammalinear}), (\ref{eqn-bc1beta}) with initial condition $\gamma_0\in L^2([0,1])$. Then, for every $\lambda>0$, there exists $c>0$ such that
\begin{equation}
\Vert \gamma(\cdot,t) \Vert_{L^2} \leq c\, \mathrm{e}^{-\lambda t} \Vert \gamma_0 \Vert_{L^2}.
\end{equation}
In fact, the equilibrium $\gamma \equiv 0$ is reached in finite time $t=t_F$, where $t_F$ is given by
\begin{equation}
t_F=\int_0^1 \left(\frac{1}{\epsilon_1(\xi)}+\frac{1}{\epsilon_2(\xi)} \right) d\xi.\label{eqn-tF}
\end{equation}
\end{proposition}
\begin{proof}
Define
\begin{equation}\label{eqn-Ddef}
D(x)=\left[ \begin{array}{cc}
A \frac{\mathrm{e}^{-\mu x}}{\epsilon_1(x)} & 0
\\
0 & B \frac{\mathrm{e}^{\mu x}}{\epsilon_2(x)}
\end{array}\right],
\end{equation}
where $A,B,\mu>0$ will be computed later.
Select:
\begin{equation}\label{eqn-lyaplinear}
V_1= \int_0^1\gamma^T(x,t) D(x) \gamma(x,t)dx.
\end{equation}
Notice that $\sqrt{V_1}$ defines a norm equivalent to $\Vert \gamma(\cdot,t) \Vert_{L^2}$. Computing the derivative of $V_1$ and integrating by parts, we obtain
\begin{eqnarray}
\dot V_1&=&
 -\int_0^1\gamma^T(x,t) \left(D(x)\Sigma(x)\right)_x \gamma(x,t)dx
+\left[ \gamma^T(x,t) D(x)\Sigma(x) \gamma(x,t)\right]_0^1,\label{eqn-Udotlinear1}
\end{eqnarray}
where we have used that  $\Sigma(x)$ and $D(x)$ commute. Since
\begin{equation}
\left(D(x)\Sigma(x)\right)_x
=
\mu
\left[ \begin{array}{cc}
A \mathrm{e}^{-\mu x} & 0
\\
0 & B\mathrm{e}^{\mu x}
\end{array}\right]>0,
\end{equation}
and, on the other hand,
\begin{eqnarray}
\left[ \gamma^T(x,t) D(x)\Sigma(x) \gamma(x,t)\right]_0^1
=
-A \alpha^2(1,t)\mathrm{e}^{-\mu}
-(B-q^2A)\beta^2(0,t),
\end{eqnarray}
choosing $B=q^2 A+\lambda_2$, $A=\lambda_2\mathrm{e}^{\mu}$, and $\mu=\lambda_1\bar \epsilon$, where $\bar \epsilon= \max_{x\in[0,1]}\left\{\frac{1}{\epsilon_1(x)}, \frac{1}{\epsilon_2(x)} \right\}$, we get that $\left(D(x)\Sigma(x)\right)_x\geq \lambda_1 D(x)$, therefore:
\begin{equation}
\dot V_1 \leq  -\lambda_1 V_1 -\lambda_2 \left( \alpha^2(1,t)+ \beta^2(0,t) \right),\label{eqn-Udotlinear2}
\end{equation}
where $\lambda_1,\lambda_2>0$ can be chosen as large as desired. This shows exponential stability of the origin for the $\gamma$ system.

To show finite-time convergence to the origin, one can find the explicit solution of (\ref{eqn-gammalinear}) as follows. Define first
\begin{equation}\label{eqn-phi}
\phi_1(x)=\int_0^x \frac{1}{\epsilon_1(\xi)} d\xi,\quad \phi_2(x)=\int_0^x \frac{1}{\epsilon_2(\xi)} d\xi,
\end{equation}
noting that they are monotonically increasing functions of $x$, and thus invertible. Note that the components of $\gamma$ verify the differential equations
\begin{eqnarray}
\alpha_t&=&-\epsilon_1(x) \alpha_x,\label{eqn-alpha}\\
\beta_t&=&\epsilon_2(x) \beta_x ,\label{eqn-beta}
\end{eqnarray}
which can be rewritten as follows
\begin{eqnarray}
\frac{\partial}{\partial t} \alpha(\phi_1^{-1}(x),t)+\frac{\partial}{\partial x} \alpha(\phi_1^{-1}(x),t)&=&0,\label{eqn-alpha1}\\
\frac{\partial}{ \partial t} \beta(\phi_2^{-1}(x),t)-\frac{\partial}{\partial x} \beta(\phi_2^{-1}(x),t)&=&0.\label{eqn-beta1}
\end{eqnarray}
The solution of these equations is $\alpha(x,t)=F_\alpha(\phi_1(x)-t)$ and $\beta(x,t)=F_\beta(\phi_2(x)+t)$, where $F_\alpha$ and $F_\beta$ are arbitrary functions.
Now, if $\alpha_0(x)$, $\beta_0(x)$ are the initial condition for the states, one obtains $F_\alpha(x)=\alpha_0(\phi_1^{-1}(x))$ (valid for $0<x<\phi_1(1)$) and $F_\beta(x)=\beta_0(\phi_2^{-1}(x))$ (valid for $0<x<\phi_2(1)$). Using the boundary conditions (\ref{eqn-bc1beta}) one finds the remaining values of $F_\alpha$ and $F_\beta$, and thus the solution of the system, as follows:
\begin{eqnarray}\label{eqn-expalpha}
\alpha(x,t)&=&
\left\{
\begin{array}{cl}
\alpha_0\left(\phi_1^{-1}\left(\phi_1(x)-t\right)\right)  &      t\leq \phi_1(x),\\
q\beta(0,t-\phi_1(x))  &     t\geq \phi_1(x),
\end{array}
\right.
\\
\beta(x,t)&
=&
\left\{
\begin{array}{cc}
\beta_0\left(\phi_2^{-1}\left(\phi_2(x)+t\right)\right)  &     t\leq \phi_2(1)-\phi_2(x),\\
0  &     t\geq \phi_2(1)-\phi_2(x),
\end{array}
\right.\label{eqn-expbeta}
\end{eqnarray}
Thus, after $t=t_F$, where
\begin{equation}
t_F=\phi_1(1)+\phi_2(1)=\int_0^1 \left(\frac{1}{\epsilon_1(\xi)}+\frac{1}{\epsilon_2(\xi)} \right) d\xi,
\end{equation}
one has that $\alpha\equiv\beta\equiv0$.
\end{proof}

\subsection{Backstepping transformation and kernel equations}
To map the original system (\ref{eqn-wlinear}) into the target system (\ref{eqn-gammalinear}), we use the following transformation:

\begin{equation}
\gamma= w-\int_0^x K(x,\xi)  w(\xi,t)d\xi, \label{eqn-tran}
\end{equation}
where
\begin{eqnarray}
K(x,\xi)&=&
\left(
\begin{array}{cc}
K^{uu}(x,\xi)  & K^{uv}(x,\xi)     \\
K^{vu}(x,\xi)  & K^{vv}(x,\xi)
\end{array}
\right)
\end{eqnarray}
is a matrix of kernels. Defining
\begin{eqnarray}
Q_0&=&
\left(
\begin{array}{cc}
0  & q     \\
0 & 1
\end{array}
\right),
\end{eqnarray}
the original and target boundary conditions (respectively (\ref{eqn-bculinear}) and (\ref{eqn-bc1beta})) can be written compactly (omitting dependences in $x$ and $t$) as
\begin{equation}
 w(0,t)= Q_0 w(0,t),\quad w(1,t) =\left( \begin{array}{c} 0 \\ U \end{array} \right),\quad
\gamma(0,t)=Q_0  \gamma(0,t),\quad  \gamma (1,t) ={0}.
\end{equation}
Introducing (\ref{eqn-tran}) into (\ref{eqn-gammalinear}), applying (\ref{eqn-wlinear}), integrating by parts and using the boundary conditions, we obtain that the original system (\ref{eqn-wlinear}) is mapped into the target system (\ref{eqn-gammalinear}) if and only if one has the following three matrix equations:
\begin{eqnarray}
0&=&C(x)+\Sigma(x) K(x,x)-K(x,x)\Sigma(x),\label{eqn-K1}\\
0&=&\Sigma(x) K_x(x,\xi)+K_\xi(x,\xi) \Sigma(\xi)+K(x,\xi)\Sigma'(\xi) -K(x,\xi) C(\xi),\label{eqn-K2}\\
0&=&K(x,0)\Sigma(0) Q_0.\label{eqn-K3}
\end{eqnarray}
Expanding (\ref{eqn-K1}), we get the following kernel equations:
\begin{eqnarray} \label{eqn-kuu}
\epsilon_1(x) K_x^{uu}+\epsilon_1(\xi) K_\xi^{uu}&=&-\epsilon_1'(\xi) K^{uu}
-c_2(\xi) K^{uv},\quad\\
\epsilon_1(x) K_x^{uv}-\epsilon_2(\xi) K_\xi^{uv}&=&\epsilon_2'(\xi) K^{uv}
-c_1(\xi) K^{uu} ,\,\,\\
\epsilon_2(x) K_x^{vu}-\epsilon_1(\xi) K_\xi^{vu}&=&
\epsilon_1'(\xi) K^{vu}
+c_2(\xi) K^{vv},\,\,\\
\epsilon_2(x) K_x^{vv}+\epsilon_2(\xi) K_\xi^{vv}&=&-\epsilon_2'(\xi) K^{vv}
+c_1(\xi) K^{vu},\,\,\label{eqn-kvv}
\end{eqnarray}
with boundary conditions obtained from (\ref{eqn-K2})--(\ref{eqn-K3})
\begin{eqnarray}
 K^{uu}(x,0)&=&\frac{\epsilon_2(0)}{q\epsilon_1(0)}K^{uv}(x,0),
 \label{eqn-bc1}\\
K^{uv}(x,x)&=&\frac{c_1(x)}{\epsilon_1(x)+\epsilon_2(x)}, \label{eqn-bc2}\\
K^{vu}(x,x)&=&-\frac{c_2(x)}{\epsilon_1(x)+\epsilon_2(x)},
 \label{eqn-bc3}\\
  K^{vv}(x,0)&=&\frac{q\epsilon_1(0)}{\epsilon_2(0)}K^{vu}(x,0). \label{eqn-bc4}
\end{eqnarray}

The equations evolve in the triangular domain $\mathcal T=\{(x,\xi):0\leq\xi \leq x \leq 1\}$. Notice that they can be written as two separate $2\times2$ hyperbolic systems, one for $K^{uu}$ and $K^{uv}$ and another for $K^{vu}$ and $K^{vv}$.

By Theorem~\ref{th-wp} (see the Appendix), one finds that, for $q\neq 0$, under the assumption that $c_1(x),c_2(x)$ are $\mathcal C([0,1])$, $\epsilon_1(x),\epsilon_2(x)$ are $\mathcal C^1([0,1])$ and that $\epsilon_1(x),\epsilon_2(x)>0$, there is a unique solution to (\ref{eqn-kuu})--(\ref{eqn-bc4}), which is in $\mathcal C(\mathcal T)$.

\subsection{The inverse transformation}
To study the invertibility of transformation (\ref{eqn-tran}), we look for a transformation of the  the target system  (\ref{eqn-gammalinear}) into the original system  (\ref{eqn-wlinear}) as follows:
\begin{eqnarray}
 w(x,t)&=&\gamma(x,t)+\int_0^x L(x,\xi) \gamma(\xi,t) d\xi,\label{eqn-traninv}
\end{eqnarray}
where
\begin{eqnarray}
L(x,\xi)&=&
\left(
\begin{array}{cc}
L^{\alpha\alpha}(x,\xi)  & L^{\alpha\beta}(x,\xi)     \\
L^{\beta\alpha}(x,\xi)  & L^{\beta\beta}(x,\xi)
\end{array}
\right).
\end{eqnarray}
Introducing (\ref{eqn-traninv}) into (\ref{eqn-wlinear}), applying (\ref{eqn-gammalinear}), integrating by parts and using the boundary conditions, we obtain as before a set of kernel equations:
\begin{eqnarray}
\epsilon_1(x) L_x^{\alpha\alpha}+\epsilon_1(\xi) L_\xi^{\alpha\alpha}
&=&
-\epsilon_1'(\xi) L^{\alpha\alpha}
+c_1(x) L^{\beta\alpha},\quad\\
\epsilon_1(x) L_x^{\alpha\beta}-\epsilon_2(\xi) L_\xi^{\alpha
  \beta}
  &=&
  \epsilon_2'(\xi) L^{\alpha
  \beta}
+c_1(x) L^{\beta\beta},
\\
\epsilon_2(x) L_x^{\beta\alpha}-\epsilon_1(\xi) L_\xi^{\beta\alpha}
&=&
\epsilon_1'(\xi) L^{\beta\alpha}
-c_2(x) L^{\alpha\alpha},
\\
  \epsilon_2(x) L_x^{\beta\beta}+\epsilon_2(\xi) L_\xi^{\beta\beta}
  &=&
  -\epsilon_2'(\xi) L^{\beta\beta}
-c_2(x) L^{\alpha\beta}
,\end{eqnarray}
with boundary conditions
\begin{eqnarray}
L^{\alpha\alpha}(x,0)&=&\frac{\epsilon_2(0)}{q\epsilon_1(0)}L^{\alpha\beta}(x,0),\\
L^{\alpha\beta}(x,x)  &=&\frac{c_1(x)}{\epsilon_1(x)+\epsilon_2(x)},\\
L^{\beta\alpha}(x,x) &=&-\frac{c_2(x)}{\epsilon_1(x)+\epsilon_2(x)},\\
L^{\beta\beta}(x,0)&=&
\frac{q\epsilon_1(0)}{\epsilon_2(0)}L^{\beta\alpha}(x,0).
\end{eqnarray}

Again by Theorem~\ref{th-wp} (see the Appendix), one finds that there is a unique solution to these equations, which is $\mathcal C(\mathcal T)$.

\subsection{Control law and main result}
From the transformation (\ref{eqn-tran}) evaluated at $x=1$, one gets
\begin{eqnarray}\label{eqn-linearcontrol}
U=\int_0^1 K^{vu}(1,\xi) u(\xi,t) d\xi
+\int_0^1 K^{vv}(1,\xi) v(\xi,t) d\xi.
\end{eqnarray}
With control law (\ref{eqn-linearcontrol}) the following result holds.
\begin{theorem}\label{th-control}
Consider system (\ref{eqn-wlinear}) with boundary conditions (\ref{eqn-bculinear}), control law (\ref{eqn-linearcontrol}), and initial condition $w_0\in L^2([0,1])$. Then, for every $\lambda>0$, there exists $c>0$ such that
\begin{equation}\label{eqn-l2normth}
\Vert w(\cdot,t) \Vert_{L^2} \leq c\, \mathrm{e}^{-\lambda t} \Vert w_0 \Vert_{L^2}.
\end{equation}
In fact, the equilibrium $w\equiv 0$ is reached in finite time $t=t_F$, where $t_F$ is given by (\ref{eqn-tF}).
\end{theorem}

\begin{proof}
Since the transformation (\ref{eqn-tran}) is invertible, when applying control law (\ref{eqn-linearcontrol}) the dynamical behavior of (\ref{eqn-wlinear}) is the same as the behavior of (\ref{eqn-gammalinear}), which is well-posed from standard results and whose explicit solution and stability properties we know from Proposition~\ref{pr-target}. Thus, we obtain the explicit solutions of $w$ from the direct and inverse transformation, as follows:
\begin{equation}
 w(x,t)=\gamma^*(x,t)+\int_0^x  L(x,\xi) \gamma^*(\xi,t)d\xi,
\end{equation}
where $\gamma^*(x,t)$ is the explicit solution of the $\alpha$, $\beta$ system, given by (\ref{eqn-expalpha})--(\ref{eqn-expbeta}), with initial conditions:
\begin{equation}
 \gamma_0(x)= w_0(x)-\int_0^x  K(x,\xi) w_0(\xi)d\xi.
\end{equation}
In particular, we know that $\gamma$ goes to zero in finite time $t=t_F$, therefore $w$ also shares that property.
Finally, since the origin of the $\gamma$ system is $L^2$ exponentially stable with an arbitrary large exponential decay rate, we conclude, using the inverse transformation, that the origin of the $w$ system is also $L^2$ exponentially stable with an arbitrary large exponential decay rate. Equation (\ref{eqn-l2normth})  follows by using the inverse and direct transformations to relate the $L^2$ norms of $w$ and $\gamma$ (using the fact that the kernels of the transformations are continuous, and thus bounded, functions).
\end{proof}
\subsection{The case $q=0$}\label{sect-q0}
 If the coefficient $q$ is zero in (\ref{eqn-bculinear}), the method presented in the paper is not valid since (\ref{eqn-bc1}) would require the value of one of the control kernels to be infinity at the boundary of the domain $\mathcal T$. Similarly, if the coefficient is close to zero one still gets very large values for the kernels close to the boundary, resulting in potentially large control laws.

The method can be modified to accommodate zero or small values of $q$ by setting a slightly different target system (\ref{eqn-alpha})--(\ref{eqn-beta}), as follows:
 \begin{eqnarray}
\alpha_t&=&-\epsilon_1(x) \alpha_x+g(x) \beta(0,t),\label{eqn-alphaq0}\\
\beta_t&=&\epsilon_2(x) \beta_x ,\label{eqn-betaq0}
\end{eqnarray}
where $g(x)$ is to be obtained from the method; regardless of the value of $g(x)$, this is a cascade system
which is still $L^2$ exponentially stable and converges in finite time by the same arguments of Proposition~\ref{pr-target}, since now, using the same Lyapunov function $V_1$ defined in (\ref{eqn-lyaplinear}), we obtain
\begin{eqnarray}
\dot V_1&=&
 -\int_0^1\gamma^T(x,t) \left(D(x)\Sigma(x)\right)_x \gamma(x,t)dx
+\left[ \gamma^T(x,t) D(x)\Sigma(x) \gamma(x,t)\right]_0^1
\nonumber \\ &&
+2 \beta(0,t) \int_0^1\alpha(x,t) A \frac{\mathrm{e}^{-\mu x}}{\epsilon_1(x)} g(x)  dx
,\label{eqn-Udotlinearq0}
\end{eqnarray}
The new term (which is the last one) can be controlled by slightly modifying the coefficients of $D(x)$ in the proof of Proposition~\ref{pr-target}, obtaining the same result as before.

The kernel equations resulting from the transformation are still the same (\ref{eqn-kuu})--(\ref{eqn-kvv}), with the same boundary conditions (\ref{eqn-bc2})--(\ref{eqn-bc4}) for $K^{uv}$, $K^{vu}$, and $K^{vv}$ (which reduces to $K^{vv}(x,0)=0$ when $q=0$), but one obtains an \emph{undetermined} boundary conditions for $K^{uu}$:
\begin{eqnarray}
K^{uu}(x,0)&=&h(x),\label{eqn-bc1q0}
\end{eqnarray}
where $h(x)$ can be chosen as desired; by choosing at least a continuous function, one can apply Theorem~\ref{th-wp} and thus the kernel equations are well-posed. After $h(x)$ has been chosen and the kernels have been computed, one obtains the value of $g(x)$ as
\begin{equation}
g(x)=q\epsilon_1(0)h(x)-\epsilon_2(0)K^{uv}(x,0).
\end{equation}

Invertibility of the transformation follows as before, thus one obtains the same result of Theorem~\ref{th-control}. The non-uniqueness in (\ref{eqn-bc1q0}) gives the designer some freedom in shaping the input function $g(x)$ from $\beta$ to $\alpha$. Also note that this has no impact in the feedback law as the kernels $K^{vu}$ and $K^{vv}$ (which are the  ones appearing in (\ref{eqn-linearcontrol})) are uniquely defined and independent of the non-unique $K^{uu}$ and $K^{uv}$.

\section{Application of the linear backstepping controller to the nonlinear system}\label{sect-main}

We wish to show that the linear controller (\ref{eqn-linearcontrol}) designed using backstepping works  \emph{locally} for the nonlinear system, in terms that will be made precise.

For that, we write our quasilinear system (\ref{eqn-z}) in a form equivalent (up to linear terms) to (\ref{eqn-wlinear}). Define
\begin{eqnarray}
\varphi_1(x)&=&\mathrm{exp}\left(
\int_0^x \frac{f_{11}(s)}{\Lambda_1(s)} ds
 \right),  \quad
 \varphi_2(x)=\mathrm{exp}\left(
-\int_0^x \frac{f_{22}(s)}{\Lambda_2(s)} ds
 \right).
\end{eqnarray}
We obtain a new state variable $w$ from $z$ using the following transformation:
\begin{equation}
w(x,t)=\left[
\begin{array}{c}
u(x,t) \\ v(x,t) \end{array}
 \right]=\left[
\begin{array}{cc}
\varphi_1(x) & 0\\
0 & \varphi_2(x)
\end{array}
 \right] \left[
\begin{array}{c}
z_1(x,t)\\
z_2(x,t)
\end{array}
 \right]=\Phi(x) z(x,t),\label{eqn-zw}
 \end{equation}
 so that
 \begin{equation}
z(x,t)=\left[
\begin{array}{cc}
\dfrac{1}{\varphi_1(x)} & 0\\
0 & \dfrac{1}{\varphi_2(x) }
\end{array}
 \right] w(x,t)=\Phi^{-1}(x) w(x,t).
 \end{equation}
It follows that $w$ verifies the following equation:
\begin{equation}
w_t+\bar \Lambda(w,x)w_x+\bar f(w,x)=0,\label{eqn-w}
\end{equation}
where
\begin{eqnarray}
\bar \Lambda(w,x)&=&\Phi(x)\Lambda(\Phi^{-1}(x) w,x)\Phi^{-1}(x),\\
\bar f(w,x)&=&\Phi(x)f(\Phi^{-1}(x) w,x)
+\bar \Lambda(w,x)
\left[
\begin{array}{cc}
-\dfrac{f_{11}(x)}{\Lambda_1(x)} & 0\\
0 & \dfrac{f_{22}(x)}{\Lambda_2(x)}
\end{array}
 \right]
 w.
\end{eqnarray}
It is evident that $\bar \Lambda(0,x)=\Phi(x)\Lambda(0,x)\Phi^{-1}(x)=\Lambda(0,x)$ and that $\bar f(0,x)=0$. Also,
\begin{equation}
C(x)=-\left.\frac{\partial \bar f(w,x)}{\partial w}\right|_{w=0}
=
\left[
\begin{array}{cc}
0 & -f_{12}(x) \\
-f_{21}(x) & 0
\end{array}
 \right].\label{eqn-nonlinearC}
\end{equation}
Thus, it is possible to write (\ref{eqn-w}) as a linear system with the same structure as (\ref{eqn-wlinear}) plus nonlinear terms:
\begin{equation}
w_t-\Sigma(x)w_x-C(x)w+\Lambda_{NL}(w,x)w_x+f_{NL}(w,x)=0,\label{eqn-w2}
\end{equation}
where
\begin{equation}
\Sigma(x)=-\Lambda(0,x)
\label{eqn-nonlinearSigma},
\end{equation}
 and
\begin{eqnarray}
\Lambda_{NL}(w,x)&=&\bar \Lambda(w,x)+\Sigma(x),\quad
f_{NL}(w,x)=\bar f(w,x)+C(x)w.
\end{eqnarray}

Computing the boundary conditions of (\ref{eqn-w2}) by combining (\ref{eqn-bcz12}) with the transformation (\ref{eqn-zw}), and defining $q=\left.\frac{\partial G_0(v)}{\partial v}\right|_{v=0}$ and $G_{NL}(v)=G_0(v)-qv$, one obtains
\begin{equation}\label{eqn-bcu2}
u(0,t)=qv(0,t)+G_{NL}(v(0,t)),\,v(1,t)=\bar U(t),
\end{equation}
where $U(t)=\varphi_2(1)U(t)$. In what follows we will consider the case $q\neq0$; the case $q=0$ is analogous (see Remark~\ref{rem-qzero}).

Notice that the linear parts of (\ref{eqn-w2}) and (\ref{eqn-bcu2}) are identical to (\ref{eqn-wlinear}) and (\ref{eqn-bculinear}), and that the coefficients $C(x)$ and $\Sigma(x)$ verify the assumptions of Section~\ref{sect-linear}. Also, it is clear that the nonlinear terms verify $\Lambda_{NL}(0,x)=0$, $f_{NL}(0,x)=\frac{\partial f_{NL}}{\partial w}(0,x)=0$, and $G_{NL}(0)=\frac{\partial G_{NL}}{\partial w}(0)=0$

Therefore, we consider using the feedback law:
\begin{eqnarray}\label{eqn-nonlinearcontrol1}
\bar U=\int_0^1 K^{vu}(1,\xi) u(\xi,t) d\xi
+\int_0^1 K^{vv}(1,\xi) v(\xi,t) d\xi,
\end{eqnarray}
which implies, in terms of the original $z$ variable:
\begin{eqnarray}\label{eqn-nonlinearcontrol2}
U=\frac{1}{\varphi_2(1)} \left(\int_0^1 K^{vu}(1,\xi) z_1(\xi,t)\varphi_1(\xi) d\xi
+ \int_0^1 K^{vv}(1,\xi)z_2(\xi,t)\varphi_2(\xi) d\xi\right),
\end{eqnarray}
where the kernels are computed from  (\ref{eqn-kuu})--(\ref{eqn-bc4}) using the coefficients $C(x)$ and $\Sigma(x)$ from (\ref{eqn-nonlinearC}) and (\ref{eqn-nonlinearSigma}).

Next, we show that the control law (\ref{eqn-nonlinearcontrol2}), which is computed for the linear part of the system, asymptotically stabilizes the nonlinear system, although locally. However, the right space to prove stability of the closed-loop system is $H^2$, instead of the space $L^2$ that was used in Section~\ref{sect-linear} for the linear system.

Denoting:
\begin{eqnarray}
q_0=\left[\begin{array}{c} 1 \\ 0 \end{array} \right],G(z)=G_0(z_2),q_1=\left[\begin{array}{c} 0 \\ 1 \end{array} \right],k(x)= \left[\begin{array}{c} \frac{\varphi_1(x)K^{vu}(1,x)}{\varphi_2(1)} \\ \frac{\varphi_2(x)K^{vv}(1,x)}{\varphi_2(1)}  \end{array} \right],
\end{eqnarray}
the boundary conditions of the closed loop system would be written as:
\begin{equation}
q_0^T z(0,t)=G(z(0,t)),\,\,q_1^T z(1,t)=\int_0^1 k^T(\xi) z(\xi,t) d\xi.\label{eqn-bcz}
\end{equation}
A necessary condition for system (\ref{eqn-z}) with boundary conditions (\ref{eqn-bcz}) to be well-posed in the space $H^2$ is that the initial conditions verify the corresponding second-order compatibility condition. These are
\begin{eqnarray}\label{eqn-cc1}
0&=&G(z_0(0))-q_0^Tz_0(0),\\
0&=&\int_0^1 k^T(\xi) z_0(\xi) d\xi- q_1^Tz_0(1),\label{eqn-cc2}\\
0&=&G'(z_0(0))\left( \Lambda(z_0(0),0)z_0'(0)+f(z_0(0),0) \right)
\nonumber\\ && -q_0^T\left( \Lambda(z_0(0),0)z_0'(0)+f(z_0(0),0) \right),\label{eqn-cc3}\\
0&=&\int_0^1 k^T(\xi) \left( \Lambda(z_0(
\xi),\xi)z_0'(\xi)+f(z_0(\xi),\xi) \right)
d\xi\nonumber\\ && -
q_1^T\left( \Lambda(z_0(1),1)z_0'(1)+f(z_0(1),1)\right).\label{eqn-cc4}
\end{eqnarray}
While (\ref{eqn-cc1}) and (\ref{eqn-cc3}) are natural compatibility conditions, the conditions  (\ref{eqn-cc2}) and (\ref{eqn-cc4}) are artificial (since they show up due to the feedback law that has been designed) and rather stringent, as they require very specific values of the initial conditions. Thus, we modify our control law in a way that, without losing its stabilizing character, does not require any specific values in the initial values beyond the natural conditions (\ref{eqn-cc1}) and (\ref{eqn-cc3}). The modification in the boundary conditions consists in adding a dynamic extension to the controller as follows:
\begin{equation}
q_0^T z(0,t)=G(z(0,t)),\,\,q_1^T z(1,t)=\int_0^1 k^T(\xi) z(\xi,t) d\xi+a(t)+b(t),
\label{eqn-bcz2}
\end{equation}
where $a(t)$ is one of the states of the following system:
\begin{gather}
\label{syst-a-b}
\dot a =-d_1 a, \, \dot b=-d_2b,
\end{gather}
where the constants $d_1$ and $d_2$ can be chosen as desired with the only conditions that $d_1,d_2>0$ and $d_1\neq d_2$. It is evident that with positive values of  these constants, (\ref{syst-a-b}) is always stable. The initial conditions of $a(t)$ and $b(t)$ are an additional degree of freedom that can be used to eliminate the compatibility conditions (\ref{eqn-cc2}) and (\ref{eqn-cc4}). With the modification of the control law, these compatibility conditions are now
\begin{eqnarray}
0&=& \int_0^1 k^T(\xi) z_0(\xi) d\xi+a(0)+b(0)-q_1^Tz_0(1),\label{eqn-cc21}\\
0&=& \int_0^1 k^T(\xi) \left( \Lambda(z_0(
\xi),\xi)z_{0_x}(\xi)+f(z_0(\xi),\xi) \right)
d\xi-d_1 a(0)-d_2 b(0)
\nonumber \\ &&
-q_1^T\left( \Lambda(z_0(1),1)z_{0_x}(1)+f(z_0(1),1)\right).\label{eqn-cc41}
\end{eqnarray}
Call
\begin{eqnarray}
P_1(z_0)&=&q_1^Tz_0(1)-\int_0^1 k^T(\xi) z_0(\xi) d\xi,\\
P_2(z_0)&=&q_1^T\left( \Lambda(z_0(1),1)z_{0_x}(1)+f(z_0(1),1)\right)
\nonumber \\ &&
-\int_0^1 k^T(\xi) \left( \Lambda(z_0(
\xi),\xi)z_{0_x}(\xi)+f(z_0(\xi),\xi) \right)
d\xi.
\end{eqnarray}
Selecting
\begin{eqnarray}
a(0)&=&  -\frac{P_2(z_0)+d_2P_1(z_0)}{d_1-d_2},\quad
b(0)= \frac{d_1P_1(z_0)+P_2(z_0)}{d_1-d_2},\label{eqn-b0}
\end{eqnarray}
the compatibility conditions are automatically verified.

We are now ready to state our main result. Define the norms $\Vert z(\cdot,t) \Vert_{H^1}= \Vert z(\cdot,t) \Vert_{L^2}+\Vert z_x(\cdot,t) \Vert_{L^2}$ and $\Vert z(\cdot,t) \Vert_{H^2}= \Vert z(\cdot,t) \Vert_{H^1}+\Vert z_{xx}(\cdot,t) \Vert_{L^2}$.
\begin{theorem}\label{thm-main}
Consider system (\ref{eqn-z}) and (\ref{syst-a-b}) with boundary conditions (\ref{eqn-bcz2})
and initial conditions $z_0=[z_{0_1}\ z_{0_2}]^T \in H^2([0,1])$, and $a(0)$ and $b(0)$ verifying (\ref{eqn-b0}), with the kernels $K^{vu}$ and $K^{vv}$ obtained from (\ref{eqn-kuu})--(\ref{eqn-bc4}) where the coefficients $C(x)$ and $\Sigma(x)$ are computed from (\ref{eqn-nonlinearC}) and (\ref{eqn-nonlinearSigma}). Then, under the assumptions of smoothness for the coefficients stated in Section~\ref{sect-prob}, for every $\lambda>0$, there exist $\delta>0$ and $c>0$
such that such that, if $\Vert z_0 \Vert_{H^2} \leq \delta$ and if the compatibility conditions  (\ref{eqn-cc1}) and (\ref{eqn-cc3}) are verified,
then:
\begin{equation}
\Vert z(\cdot,t) \Vert_{H^2}^2+a^2(t)+b^2(t)\leq c\,\mathrm{e}^{-\lambda t}\left( \Vert z_0 \Vert_{H^2}^2+a^2(0)+b^2(0)\right).
\end{equation}
\end{theorem}

\section{Proof of Theorem~\ref{thm-main}}\label{sect-proof}
\subsection{Preliminary definitions}
We first establish some definitions and notation. For $\gamma(x)\in\mathbb R^2$  with components $\alpha(x)$ and $\beta(x)$ denote $\vert \gamma(x) \vert =\vert \alpha(x) \vert + \vert \beta(x) \vert$, and
\begin{equation}
\Vert \gamma \Vert_{\infty} = \sup_{x\in[0,1]} \vert \gamma(x) \vert,\quad
\Vert \gamma \Vert_{L^1}= \int_0^1 \vert \gamma(\xi) \vert d\xi.
\end{equation}
In what follows, for a time-varying vector $\gamma(x,t)$, we denote $\vert \gamma \vert =\vert \gamma(x,t) \vert$ and $\Vert \gamma \Vert=\Vert \gamma(\cdot,t) \Vert$ to simplify our notation.
For a $2 \times 2$ matrix $M$, denote:
\begin{equation}
\vert M \vert = \max \{ \vert M\gamma \vert; \gamma \in\mathbb R^2,\vert \gamma \vert =1\}.
\end{equation}

For the kernel matrices $K(x,\xi)$ and $L(x,\xi)$ denote
\begin{equation}
\Vert K \Vert_{\infty}=\sup_{(x,\xi)\in\mathcal T} \vert K(x,\xi)\vert.
\end{equation}
For $\gamma\in H^2([0,1])$, recall the following well-known inequalities, that will be used later:
\begin{eqnarray}
\Vert \gamma \Vert_{L^1} &\leq& C_1 \Vert \gamma \Vert_{L^2} \leq C_2 \Vert \gamma \Vert_{\infty},\\
\Vert \gamma \Vert_{\infty} &\leq& C_3\left[ \Vert \gamma \Vert_{L^2}+\Vert \gamma_x \Vert_{L^2}\right]\leq C_4 \Vert \gamma \Vert_{H^1},\label{eqn-sobolev1}\\
\Vert \gamma_x \Vert_{\infty} &\leq& C_5\left[ \Vert \gamma_x \Vert_{L^2}+\Vert \gamma_{xx} \Vert_{L^2}\right]\leq C_6 \Vert \gamma \Vert_{H^2}.\quad\,\,\label{eqn-sobolev2}
\end{eqnarray}

 Define the following linear functionals, the first two of which are, respectively, the inverse and direct transformations (\ref{eqn-tran}) and (\ref{eqn-traninv}):
\begin{eqnarray}\label{eqn-calK}
\mathcal K[\gamma](x)
&=&\gamma(x,t)-\int_0^x K(x,\xi) \gamma(\xi,t) d\xi,\\
\mathcal L[\gamma](x) \label{eqn-calL}
&=&\gamma(x,t)+\int_0^x L(x,\xi) \gamma(\xi,t) d\xi,\\
\mathcal K_{1}[\gamma](x)
&=&-K(x,x)\gamma(x,t)+\int_0^x K_{\xi}(x,\xi) \gamma(\xi,t) d\xi,\qquad\\
\mathcal K_{2}[\gamma](x)
&=&-K(x,x)\gamma(x,t)-\int_0^x K_{x}(x,\xi) \gamma(\xi,t) d\xi,\qquad\\
\mathcal L_{1}[\gamma](x)
&=&L(x,x)\gamma(x,t)+\int_0^x L_x(x,\xi) \gamma(\xi,t) d\xi,\qquad\\
\mathcal L_{11}[\gamma](x)
&=&(L_x(x,x)+L_\xi(x,x))\gamma(x,t)+L_x(x,x)\gamma(x,t)
\nonumber \\ &&
+\int_0^x L_{xx}(x,\xi) \gamma(\xi,t) d\xi.\label{eqn-calLx}
\end{eqnarray}
For simplicity, in what follows we drop writing the $x$ dependence in functionals and the $t$ dependence in the variables.

Using (\ref{eqn-calK}) and (\ref{eqn-calL}), we define $F_1[\gamma]$ and $F_2[\gamma]$ as:
\begin{eqnarray}\label{eqn-F1}
F_1&=&\Lambda_{NL}\left(\mathcal L[\gamma],x\right),\quad
F_2=f_{NL}\left(\mathcal L[\gamma],x\right).
\end{eqnarray}

To prove Theorem~\ref{thm-main}, we notice that if we apply the (invertible) backstepping transformation (\ref{eqn-tran}) to the nonlinear system (\ref{eqn-w2}) we obtain the following transformed system:
\begin{eqnarray}
0&=&\gamma_t-\Sigma(x)  \gamma_x
+\Lambda_{NL}(w,x)w_x+f_{NL}(w,x)
\nonumber \\ &&
+\int_0^x K(x,\xi) \left(\Lambda_{NL}(w,\xi)w_x(\xi)
-f_{NL}(w,\xi) \right) d\xi,\quad\label{eqn-gammaNL}
\end{eqnarray}
and using the inverse transformation (\ref{eqn-traninv}) the equation can be expressed fully in terms of $\gamma$ as:
\begin{eqnarray}\label{eqn-gammaNL2}
\gamma_t-\Sigma(x)  \gamma_x+F_3[\gamma,\gamma_x]+F_4[\gamma]=0,
\end{eqnarray}
where the functionals $F_3$ and $F_4$ are defined as
\begin{eqnarray}
F_3&=&\mathcal K\left[ F_1[\gamma]\gamma_x\right],
\label{eqn-F3}
\\
F_4&=&\mathcal K\left[ F_1[\gamma] \mathcal L_{1} \left[\gamma \right] +F_2[\gamma]\right].\label{eqn-F4}\qquad\,
\end{eqnarray}
The boundary conditions are
\begin{equation}\label{eqn-gammabc}
 \alpha(0,t)=q  \beta(0,t)+G_{NL}(\beta(0,t)),\,\  \beta (1,t) =a(t)+b(t).
\end{equation}
By the assumptions on the coefficients and applying Theorem~\ref{th-sm}, the direct and inverse transformations (\ref{eqn-tran}) and (\ref{eqn-traninv}) have kernels that are $\mathcal C^2 (\mathcal T)$ functions. Differentiating twice with respect to $x$ in these transformations, it can be shown that the $H^2$ norm of $\gamma$ is equivalent to the $H^2$ norm of $z$ (see for instance~\cite{vazquez-coron}). Thus, if we show $H^2$ local stability of the origin for (\ref{eqn-gammaNL2})--(\ref{eqn-gammabc}), the same holds for $z$.

We proceed by analyzing (using a Lyapunov function) the growth of $\Vert \gamma\Vert_{L^2}$, $\Vert \gamma_t\Vert_{L^2}$ and $\Vert \gamma_{tt}\Vert_{L^2}$. Relating these norms with $\Vert \gamma \Vert_{H^2}$, we then prove $H^2$ local stability for $\gamma$.

\subsection{Analyzing the growth of $\Vert \gamma\Vert_{L^2}$}
Define
\begin{equation}
V_1= \int_0^1\gamma^T(x,t) D(x) \gamma(x,t)dx,
\end{equation}
for $D(x)$ as in (\ref{eqn-Ddef}). Proceeding analogously to (\ref{eqn-Udotlinear1})--(\ref{eqn-Udotlinear2}), we get some extra nonlinear terms:
\begin{eqnarray}
\dot V_1 &= & -\int_0^1\gamma^T(x,t) \left(D(x)\Sigma(x)\right)_x \gamma(x,t)dx
+\left[ \gamma^T(x,t) D(x)\Sigma(x) \gamma(x,t)\right]_0^1
\nonumber \\ &&
- 2  \int_0^1\gamma^T(x,t) D(x)\left(F_3[\gamma,\gamma_x]+ F_4[\gamma] \right) dx
.
\end{eqnarray}
Let us analyze first the last term:
\begin{equation}
 2 \left| \int_0^1\gamma^T(x,t) D(x)\left(F_3[\gamma,\gamma_x]+ F_4[\gamma] \right)  dx \right|
\leq
K_1 \int_0^1 \vert \gamma \vert \left( \vert F_3[\gamma,\gamma_x]\vert
+  \vert F_4[\gamma] \vert\right) dx.
\end{equation}

Applying Lemma~\ref{lem-F3F4bound} (see the Appendix), we obtain that there exists a $\delta_1$, such that for $\Vert \gamma \Vert_{\infty} < \delta_1$,
\begin{eqnarray}
\int_0^1 \vert \gamma \vert \vert F_3[\gamma,\gamma_x] \vert dx &\leq& K_2 \Vert \gamma_x \Vert_{\infty}\Vert \gamma \Vert^2_{L^2},\quad\\
 \int_0^1 \vert \gamma \vert \vert F_4[\gamma] \vert dx &\leq& K_3 \Vert \gamma \Vert_{\infty} \Vert \gamma \Vert_{L^2}^2,
\end{eqnarray}
and using inequality (\ref{eqn-sobolev1}) and noting that $\Vert \gamma \Vert_{L^2} \leq K_4 V_1^{1/2}$, we obtain
\begin{eqnarray}
\int_0^1 \vert \gamma \vert \vert F_3[\gamma,\gamma_x] \vert dx &\leq& K_5 \Vert \gamma_x \Vert_{\infty}V_1,\quad\\
 \int_0^1 \vert \gamma \vert \vert F_4[\gamma] \vert dx &\leq& K_6  \Vert \gamma_x \Vert_{\infty}V_1+K_7 V_1^{3/2}.
\end{eqnarray}
Now,
\begin{eqnarray}
\left[ \gamma^T(x,t) D(x)\Sigma(x) \gamma(x,t)\right]_0^1
&=&B a^2(t)\mathrm{e}^{\mu}
-A \alpha^2(1,t)\mathrm{e}^{-\mu}
-B \beta^2(0,t)
\nonumber \\ &&
+A(q\beta(0,t)+G_{NL}(\beta(0,t)))^2,
\end{eqnarray}
and for $\Vert \gamma \Vert_{\infty} < \delta_1$, $\vert G_{NL}(\beta(0,t))\vert \leq K_8 \vert \beta(0,t)\vert $, and $A>0$, we obtain
\begin{eqnarray}
\left[ \gamma^T(x,t) D(x)\Sigma(x) \gamma(x,t)\right]_0^1
&\leq &
-A \alpha^2(1,t)\mathrm{e}^{-\mu}
+(A(\vert q\vert +K_8)^2-B) \beta^2(0,t)
\nonumber \\ &&
+B a^2(t)\mathrm{e}^{\mu}
\end{eqnarray}
Thus,
choosing $B=(\vert q\vert +K_8)^2 A+\lambda_2$ and $A$ and $\mu$ as in the proof of Proposition~\ref{pr-target}, we obtain the following proposition:
\begin{proposition}\label{pr-gamma}
There exists $\delta_1$ such that if $\Vert \gamma \Vert_{\infty} < \delta_1$ then
\begin{eqnarray}
\dot V_1 &\leq&  -\lambda_1 V_1 -\lambda_2 \left( \alpha^2(1,t)+ \beta^2(0,t) \right)
+ C_1 V_1^{3/2}  +C_2
\Vert \gamma_x \Vert_{\infty}
V_1
\nonumber \\ &&
+C_3 (a^2(t)+b^2(t)),
\end{eqnarray}
where $\lambda_1$, $\lambda_2$, $C_1$, $C_2$ and $C_3$ are positive constants.
\end{proposition}

\subsection{Analyzing the growth of $\Vert \gamma_t\Vert_{L^2}$}
Define $\eta=\gamma_t$. Notice that the norms of $\eta$ and $\gamma_x$ are related (see Lemma~\ref{lem-gammaetaequiv} in the Appendix). Taking a partial derivative in $t$ in (\ref{eqn-gammaNL}) we obtain an equation for $\eta$ as follows:
\begin{eqnarray}\label{eqn-etaNL2}
 \eta_t+\left(F_1[\gamma]-\Sigma(x)\right)  \eta_x
+F_5[\gamma,\gamma_x,\eta]+F_6[\gamma,\eta]=0,
\end{eqnarray}
where $F_5$ and $F_6$ are defined as
\begin{eqnarray}
F_5&=& \mathcal K_{1} \left[ F_1[\gamma]\eta \right]
+ \int_0^x K(x,\xi) F_{12}[\gamma,\gamma_x]\eta(\xi) d\xi
+K(x,0) \Lambda_{NL}\left(\gamma(0),0\right) \eta(0)
\nonumber \\ &&
 +\mathcal  K \left[ F_{11}[\gamma,\eta]\gamma_x\right]
,\label{eqn-F5}
\\
F_6&=&\mathcal  K\left[F_{11}[\gamma,\eta] \mathcal  L_x \left[\gamma\right] \right]
+\mathcal  K\left[F_{1}[\gamma]\mathcal  L_x \left[\eta\right]\right]
+\mathcal K[F_{21} [\gamma,\eta]],\label{eqn-F6}
\end{eqnarray}
where
\begin{eqnarray}
F_{11}&=&
\frac{\partial \Lambda_{NL}}{\partial \gamma} \left(\mathcal  L[\gamma],x\right)\mathcal  L[\eta],\quad\,\,\,
\\
F_{12}&=&
\frac{\partial \Lambda_{NL}}{\partial \gamma} \left(\mathcal  L[\gamma],x\right) \left(\gamma_x+ \mathcal  L_x[ \gamma] \right)
+\frac{\partial \Lambda_{NL}}{\partial x} \left(\mathcal  L[\gamma],x\right) ,\quad\,\,\,
\\
F_{21}&=&\frac{\partial f_{NL}}{\partial \gamma} \left(\mathcal  L[\gamma],x\right)L[\eta].\quad
\end{eqnarray}
The boundary conditions for $\eta=[\eta_1 \ \eta_2]^T$ are
\begin{equation}
 \eta_1(0,t)=q  \eta_2(0,t)
 +G'_{NL}(\beta(0,t))\eta_2(0,t)
 ,\,\  \eta_2 (1,t) =-d_1a(t)-d_2 b(t).
\end{equation}

To find a Lyapunov function for $\eta$, we use the next lemma:
\begin{lemma}\label{lem-lyapeta}
There exists $\delta>0$ such that, for $\Vert \gamma \Vert_{\infty}<\delta$, there exists a symmetric matrix $R[\gamma]>0$ verifying the identity:
\begin{equation}
R[\gamma] \left(\Sigma(x)-F_1[\gamma]\right)- \left(\Sigma(x)-F_1[\gamma]\right)^T R[\gamma]=0,\label{eqn-symm}
\end{equation}
and the following bounds:
\begin{eqnarray}
R[\gamma](x)&\leq&c_1+c_2 \Vert \gamma \Vert_{\infty},\label{eqn-Rpos2}\\
\vert \left( \left(R[\gamma]-D(x)\right) \Sigma(x)\right)_x \vert&\leq& c_2 \Vert \gamma \Vert_{\infty}  \left(1 +   \Vert \gamma _x \Vert_\infty  \right),\quad\label{eqn-Thetax}\\
\vert \left( R[\gamma] \right)_t \vert &\leq& c_3 \left(\vert \eta\vert+  \Vert \eta \Vert_{L^1}   \right),\label{eqn-Thetat}
\end{eqnarray}
where $c_1,c_2,c_3$ are positive constants.
\end{lemma}

\begin{proof}
We explicitly construct $R[\gamma]$ as
\begin{equation}
\label{eqn-Rdef}
R[\gamma]=D(x)+\Theta[\gamma],
\end{equation}
with
\begin{equation}
\Theta[\gamma]=\left[
\begin{array}{cc}
0 & \psi[\gamma]
 \\
 \psi[\gamma]
&0
\end{array}
 \right],\label{eqn-Thetadef}
\end{equation}
where $\psi[\gamma]$ is defined as:
\begin{equation}
\psi[\gamma]=\frac{D_{11}(x)\left(F_1[\gamma]\right)_{12}-D_{22}(x)\left(F_1[\gamma]
\right)_{21}}{\epsilon_2(x)+\epsilon_1(x)+\left(F_1[\gamma]\right)_{11}-\left(F_1[\gamma]\right)_{22}},\label{eqn-psidef}
\end{equation}
where $\left(F_1[\gamma]\right)_{ij}$ denotes the coefficient in row $i$ and column $j$ in the matrix $F_1[\gamma]$.
Identity (\ref{eqn-symm}) follows by using the construction of $R[\gamma](x)$ in (\ref{eqn-Rdef})--(\ref{eqn-psidef}), and the fact that $D(x)$ and $\Sigma(x)$ are diagonal and commute.
To ensure that the denominator of (\ref{eqn-psidef}) is different from zero, denote $K_1=\min_{x\in[0,1]} \left(\epsilon_1(x)+\epsilon_2(x)\right)>0$. Applying (\ref{eqn-F1bound}) from Lemma~\ref{lem-F3F4bound}, there exists $\delta_1$ for which, if $\Vert \gamma \Vert_{\infty}<\delta_1$, one gets:
\begin{equation}
\epsilon_2(x)+\epsilon_1(x)+\left(F_1[\gamma]\right)_{22}-\left(F_1[\gamma]\right)_{11}
\geq K_1 - K_2 \Vert \gamma \Vert_{\infty},
\end{equation}
thus if $\Vert \gamma \Vert_{\infty} \leq  \min\{\delta_1,\delta_2\}$ with $\delta_2=\frac{ K_1}{2K_2} $, we obtain
\begin{equation}
\epsilon_2(x)+\epsilon_1(x)+\left(F_1[\gamma]\right)_{22}-\left(F_1[\gamma]\right)_{11}
\geq \frac{K_1}{2},\label{dem-bound}
\end{equation}
thus $\psi[\gamma]$ is well-defined. Applying again (\ref{eqn-F1bound}) in the numerator of (\ref{eqn-psidef}) to bound (\ref{eqn-Thetadef}), we obtain:
\begin{equation}
\Vert \Theta[\gamma] \Vert \leq K_3 \Vert \gamma \Vert_{\infty},\label{eqn-thetabound}
\end{equation}
and noting $K_4=\Vert D \Vert_{\infty}$, we obtain directly the bound (\ref{eqn-Rpos2}), and by choosing $\Vert \gamma \Vert_{\infty} \leq \min\{\delta_1,\delta_2,\delta_3\}$, with $\delta_3= \frac{K_4}{2K_3}$, we show $R[\gamma]>0$.

Inequality (\ref{eqn-Thetax}) is equivalent to showing:
\begin{equation}
\vert \left(\Theta[\gamma] \Sigma(x)\right)_x \vert \leq c_2 \Vert \gamma \Vert_{\infty}  \left(1 + \Vert \gamma _x \Vert_\infty  \right).
\end{equation}
We first use (\ref{eqn-thetabound}) to bound $\vert  \Theta[\gamma](x) \Sigma_x(x) \vert$, and for $\vert  \Theta_x[\gamma](x) \Sigma(x) \vert$ we take a derivative in $x$ in (\ref{eqn-psidef}), use the bound (\ref{dem-bound}) and use the fact that $\frac{\partial}{\partial x} F_1[\gamma]= F_{12}[\gamma,\gamma_x]$ and using Lemma~\ref{lem-F5F6bound}, there exists $\delta_4$ such that if $\Vert \gamma \Vert_{\infty} \leq \delta_4$,
\begin{eqnarray}
\vert F_{12}[\gamma,\gamma_x] \vert &\leq& K_1 \left( \Vert \gamma \Vert_{\infty}+\Vert \gamma_x \Vert_{\infty}\right).
\end{eqnarray}
To show (\ref{eqn-Thetat}) we use $\frac{\partial}{\partial t} F_1[\gamma]=
\vert F_{11}[\gamma,\eta]$ and apply  Lemma~\ref{lem-F5F6bound}.
Setting $\delta=\min\{\delta_1,\delta_2,\delta_3,\delta_4\}$ the lemma follows.
\end{proof}

Define:
\begin{equation}
V_2= \int_0^1\eta^T(x,t) R[\gamma](x) \eta(x,t)dx.
\end{equation}
Computing $\dot V_2$, applying Lemma~\ref{lem-lyapeta}, and integrating by parts, we find
\begin{eqnarray}
\dot V_2
&=&
-\int_0^1\eta^T(x,t)\left( R[\gamma]\left(\Sigma(x) -F_1[\gamma] \right)\right)_x \eta(x,t)dx
\nonumber \\ &&
+\left[\eta^T(x,t)R[\gamma](x)\left(\Sigma(x) -F_1[\gamma](x) \right) \eta(x,t) \right]_{x=0}^{x=1}
+\int_0^1\eta^T(x,t) \left(R[\gamma]\right)_t \eta(x,t)dx
\nonumber \\ &&
-2\int_0^1\eta^T(x,t) R[\gamma]F_5[\gamma,\gamma_x,\eta,\eta_x,]dx
-2\int_0^1\eta^T(x,t) R[\gamma]F_6[\gamma,\eta] dx.\label{eqn-Vdot}
\end{eqnarray}
The first three terms of (\ref{eqn-Vdot}) are analyzed using Lemma~\ref{lem-lyapeta}. Thus, there exists $\delta_1$ such that, for $\Vert \gamma \Vert_{\infty}<\delta$, we find, for the first term:
\begin{eqnarray}
&&
-\int_0^1\eta^T(x,t)\left( R[\gamma]\left(\Sigma(x) -F_1[\gamma] \right)\right)_x \eta(x,t)dx
\nonumber \\
&\leq &
-\lambda_1 V_2
+K_1 \Vert \eta \Vert_{L^2}^2\left( \Vert \gamma \Vert_{\infty} + \Vert \gamma _x \Vert_\infty\right). \quad
\,\,\label{eqn-Vdott1}
\end{eqnarray}
The second term of  (\ref{eqn-Vdot}) is bounded using the boundary conditions, (\ref{eqn-Udotlinear1})--(\ref{eqn-Udotlinear2}), and Lemma~\ref{lem-lyapeta}, as:
\begin{eqnarray}
&&
\left[\eta^T(x,t)R[\gamma](x)\left(\Sigma(x) -F_1[\gamma](x) \right) \eta(x,t) \right]_{x=0}^{x=1}
\nonumber \\
&\leq&
-\lambda_2 \left( \eta_1^2(1,t)+ \eta_2^2(0,t) \right)
+
K_2 \Vert \gamma \Vert_{\infty} \left( \eta_2^2(0,t) +\eta_1^2(1,t)\right)
\nonumber \\ &&
+K_3(1+ \Vert \gamma \Vert_{\infty})(a(t)^2+b(t)^2).
\end{eqnarray}
Finally, we bound the third term of (\ref{eqn-Vdot}) applying Lemma~\ref{lem-lyapeta} as follows:
\begin{eqnarray}
\int_0^1\eta^T(x,t) \left(R[\gamma]\right)_t \eta(x,t)dx
&\leq& K_3\int_0^1 \vert \eta\vert^2   \left(\vert \eta\vert+  \Vert \eta \Vert_{L^1}   \right) dx
\leq  K_4\Vert \eta \Vert^2_{L^2} \Vert \eta \Vert_{\infty}.\qquad\label{eqn-Vdott2}
\end{eqnarray}

Applying Lemmas~\ref{lem-lyapeta} and~\ref{lem-F5F6bound} to the last terms of (\ref{eqn-Vdot}), we get, for $\Vert \gamma \Vert_{\infty} < \delta$,
\begin{eqnarray}
&&
2 \left|\int_0^1\eta^T(x,t) R[\gamma]F_5[\gamma,\gamma_x,\eta,\eta_x]dx \right|
\leq
K_5 \int_0^1 \vert \eta \vert \vert F_5[\gamma,\eta] \vert dx
\nonumber \\
&\leq&
K_6 \Vert \eta \Vert^2_{L^2} \left( \Vert \gamma \Vert_{\infty}+\Vert \gamma_x \Vert_{\infty}\right)
+K_{7} \Vert \eta \Vert_{L^2}\vert \eta(0,t)\vert \vert \gamma(0,t) \vert ,
\end{eqnarray}
and
\begin{eqnarray}
&&
2 \left|\int_0^1\eta^T(x,t) R[\gamma]F_6[\gamma,\eta]dx \right|
\leq
K_{8} \int_0^1 \vert \eta \vert \vert F_6[\gamma,\eta] \vert dx
\leq K_{9}
\Vert \eta \Vert^2_{L^2} \Vert \gamma \Vert_{\infty}.
\end{eqnarray}

Thus, it is clear that by choosing $\Vert \gamma \Vert_{\infty} $ small enough, using Lemma~\ref{lem-gammaetaequiv} to bound $\Vert \gamma_x \Vert_{\infty}$ by $\Vert \eta \Vert_{\infty}$, and noting $\Vert \eta \Vert_{L^2} \leq K_{10} V^{1/2}$, we obtain the following proposition:
\begin{proposition}\label{pr-gammat}
There exists $\delta_2$ such that if $\Vert \gamma \Vert_{\infty} < \delta_2$
\begin{eqnarray}
\dot V_2 &\leq&
-\lambda_3 V_2-\lambda_4 \left( \eta_1^2(1,t)+ \eta_2^2(0,t) \right)
+K_1V_2 \Vert \eta  \Vert_{\infty}+K_2b(t)^2,\quad\,\,\,
\end{eqnarray}
for $\lambda_2,\lambda_3,K_1,K_2$ positive constants.
\end{proposition}

\subsection{Analyzing the growth of $\Vert \gamma_{tt}\Vert_{L^2}$}
Define $\theta=\eta_t$. Notice that the norms of $\theta$ and $\eta_x$ are related (see Lemma~\ref{lem-thetaetaequiv} in the Appendix). Taking a partial derivative in $t$ in (\ref{eqn-etaNL2}) we obtain an equation for $\theta$:
\begin{eqnarray}
&&
\theta_t+\left(F_1[\gamma]-\Sigma(x)\right)  \theta_x
+F_7[\gamma,\gamma_x,\eta,\eta_x,\theta]+F_8[\gamma,\eta,\theta]=0,\label{eqn-thetaNL2}
\end{eqnarray}
where $F_7$ and $F_8$ are defined as
\begin{eqnarray}
F_7&=&\mathcal K_{1} \left[ F_{11}[\gamma,\eta]\eta \right]+ \int_0^x K(x,\xi) F_{12}[\gamma,\gamma_x]\theta(\xi) d\xi
+ \mathcal K_{1} \left[ F_1[\gamma]\theta \right]
\nonumber \\ &&
+ \int_0^x K(x,\xi) F_{14}[\gamma,\gamma_x,\eta,\eta_x]\eta(\xi) d\xi
+K(x,0)\frac{\partial \Lambda_{NL}}{\partial \gamma}\left(\gamma(0),0\right)\eta(0) \eta(0)
\nonumber \\ &&
 +K(x,0) \Lambda_{NL}\left(\gamma(0),0\right) \theta(0)
 +\mathcal  K \left[ F_{11}[\gamma,\eta]\eta_x\right]
  +\mathcal  K \left[ F_{13}[\gamma,\eta,\theta]\gamma_x\right]
,\label{eqn-F7}
\\
F_8&=&2\mathcal  K\left[F_{11}[\gamma,\eta] \mathcal  L_x \left[\eta\right] \right]
+\mathcal  K\left[F_{1}[\gamma]\mathcal  L_x \left[\theta\right]\right]
+\mathcal  K\left[F_{13}[\gamma,\eta,\theta] \mathcal  L_x \left[\gamma\right] \right]
\nonumber \\ &&
+\mathcal K[F_{22} [\gamma,\eta,\theta]],\label{eqn-F8}
\end{eqnarray}
where
\begin{eqnarray}
F_{13}&=&
\frac{\partial \Lambda^2_{NL}}{\partial \gamma^2} \left(\mathcal L[\gamma],x\right) \mathcal L[\eta] \mathcal L[\eta]
+\frac{\partial \Lambda_{NL}}{\partial \gamma} \left(\mathcal L[\gamma],x\right) \mathcal L[\theta],\\
F_{14}&=&
\frac{\partial^2 \Lambda_{NL}}{\partial \gamma^2} \left(\mathcal L[\gamma],x\right) \mathcal L[\eta] \left(\gamma_x+\mathcal L_{1}[\gamma]\right)
 +\frac{\partial \Lambda_{NL}}{\partial \gamma} \left(\mathcal L[\gamma] ,x\right)
\left(\eta_x+ \mathcal L_{1}[\eta]\right)
\nonumber \\ &&
+\frac{\partial^2 \Lambda_{NL}}{\partial x \partial \gamma} \left(\mathcal L[\gamma],x\right)  \mathcal L[\eta] ,\quad\,\,\,
\\
F_{22}&=&\frac{\partial^2 f_{NL}}{\partial \gamma^2} \left(\mathcal L[\gamma],x\right)\mathcal L[\eta]\mathcal L[\eta]
+
\frac{\partial f_{NL}}{\partial \gamma} \left(\mathcal L[\gamma],x\right)\mathcal L[\theta].\quad
\end{eqnarray}
 The boundary conditions for $\theta=[\theta_1 \ \theta_2]^T$ are
\begin{eqnarray}
 \theta_1(0,t)&=&q  \theta_2(0,t) +G'_{NL}(\beta(0,t))\theta_2(0,t)+G''_{NL}(\beta(0,t))\eta^2_2(0,t),\\  \theta_2 (1,t) &=&d_1^2 a(t)+d_2^2 b(t).
\end{eqnarray}

Since (\ref{eqn-thetaNL2}) has the same structure as (\ref{eqn-etaNL2}), we define:
\begin{equation}
V_3= \int_0^1\theta^T(x,t) R[\gamma](x) \theta(x,t)dx,
\end{equation}
where $R[\gamma](x)$ was defined in Lemma~\ref{lem-lyapeta}.

Computing $\dot V_3$, and proceeding exactly as in (\ref{eqn-Vdot}), we find:
\begin{eqnarray}
\dot V_3
&=&
-\int_0^1\theta^T(x,t)\left( R[\gamma]\left(\Sigma(x) -F_1[\gamma] \right)\right)_x \theta(x,t)dx
\nonumber \\ &&
+\left[\theta^T(x,t)R[\gamma](x)\left(\Sigma(x) -F_1[\gamma](x) \right) \theta(x,t)\right]_{x=0}^{x=1}
+\int_0^1\theta^T(x,t) \left(R[\gamma]\right)_t \theta(x,t)dx
\nonumber \\ &&
-2 \int_0^1\theta^T(x,t) R[\gamma]F_7[\gamma,\gamma_x,\eta,\eta_x,\theta]dx
-2 \int_0^1\theta^T(x,t) R[\gamma]F_6[\gamma,\eta,\theta]dx.\qquad
\label{eqn-Wdot}
\end{eqnarray}
The first three terms of (\ref{eqn-Wdot}) are analyzed as in (\ref{eqn-Vdott1})--(\ref{eqn-Vdott2}):
\begin{eqnarray}
\dot V_3
&\leq&
-\lambda_1 V_3+K_1 \Vert \theta \Vert_{L^2}^2   \left(\Vert \gamma \Vert_{\infty} + \Vert \gamma _x \Vert_\infty  \right)
+\left(K_2 \Vert \gamma \Vert_{\infty} -\lambda_2\right) \left( \theta_1^2(1,t)+ \theta_2^2(0,t) \right)
\nonumber \\ &&
+2 \left|\int_0^1\theta^T(x,t) R[\gamma]F_6[\gamma,\eta,\theta]dx\right|
+2 \left|\int_0^1\theta^T(x,t) R[\gamma]F_7[\gamma,\gamma_x,\eta,\eta_x,\theta]dx\right|
\nonumber \\ &&+  K_3\Vert \theta \Vert^2_{L^2} \Vert \eta \Vert_{\infty}
+K_4(\eta^4_2(0,t)+(1+\Vert \gamma \Vert_{\infty})(a^2(t)+b^2(t)))
.\qquad\label{eqn-Wdot2}
\end{eqnarray}
Finally, applying Lemmas~\ref{lem-lyapeta} and~\ref{lem-F7F8bound} in the last two terms of (\ref{eqn-Wdot2}), there exists a $\delta$, such that for $\Vert \gamma \Vert_{\infty} < \delta$,
\begin{eqnarray}
&& 2 \left|\int_0^1\theta^T(x,t) R[\gamma]F_7[\gamma,\gamma_x,\eta,\eta_x,\theta]dx \right|
\leq
K_5 \int_0^1 \vert \theta \vert \vert F_7[\gamma,\gamma_x,\eta,\eta_x,\theta] \vert dx
\nonumber \\
&\leq&
K_6 \Vert \theta \Vert_{L^2}^2 \left( \Vert \gamma \Vert_{\infty}+\Vert \gamma_x \Vert_{\infty}\right)+
K_7\Vert \theta \Vert_{L^2} \Vert \eta \Vert_{L^2}^2
+
K_8\Vert \theta \Vert_{L^2} \Vert \eta_x \Vert_{L^2}\Vert \eta \Vert_{\infty}
+K_9 \Vert \theta \Vert_{L^2}
\nonumber \\  &&
+K_{10} \Vert \theta \Vert_{L^2}
 \left( \Vert \eta \Vert_{L^2}\Vert \eta \Vert_{\infty}^2+\vert\eta(0,t)\vert^2+ \vert \gamma(0,t)\vert \vert \theta(0,t) \vert\right),\qquad
\end{eqnarray}
and
\begin{eqnarray}
&&
2 \left|\int_0^1\theta^T(x,t) R[\gamma]F_8[\gamma,\eta,\theta]dx \right|
\leq
K_{11} \int_0^1 \vert \theta \vert \vert F_8[\gamma,\eta,\theta] \vert dx
\nonumber \\
&\leq& K_{11}
\Vert \theta \Vert^2_{L^2} \Vert \gamma \Vert_{\infty}+K_{12} \Vert \eta \Vert_{L^2} \Vert \theta \Vert_{L^2} \Vert \eta \Vert_{\infty}
+K_{12} \Vert \eta \Vert_{L^2} \Vert \theta \Vert_{L^2}^2
+K_{13} \Vert \eta \Vert_{L^2}^2 \Vert \theta \Vert_{L^2}.\qquad\,\,
\end{eqnarray}

Thus, by choosing $\Vert \gamma \Vert_{\infty} $ and $\Vert \eta \Vert_{\infty} $ small enough to apply Lemma~\ref{lem-thetaetaequiv}, we finally obtain the following proposition:
\begin{proposition}\label{pr-gammatt}
There exists $\delta_3$ such that if $\Vert \gamma \Vert_{\infty}+\Vert \eta \Vert_{\infty}< \delta_3$ then
\begin{eqnarray}
\dot V_3 &\leq&
-\lambda_5 V_3-\lambda_6 \left( \theta_1^2(1,t)+ \theta_2^2(0,t) \right)
+K_1V_3 V_2^{1/2}
+K_2 V_2 V_3^{1/2}
\nonumber \\ &&
+K_3 V_3^{3/2}+K_4\Vert \eta \Vert_{\infty} \eta^2_2(0,t)+K_5(a^2+b^2),\quad
\end{eqnarray}
where $\lambda_5,\lambda_6,K_1,K_2,K_3,K_4,K_5$ are positive constants.
\end{proposition}

\subsection{Proof of $H^2$ stability of $\gamma$}
Defining $W=V_1+V_2+V_3$, and combining Propositions~\ref{pr-gamma},~\ref{pr-gammat}, and~\ref{pr-gammatt}, there exists $\delta$ such that if $\Vert \gamma \Vert_{\infty}+\Vert \eta \Vert_{\infty}< \delta$
\begin{equation}
\dot W\leq - \lambda_1 W+C_1 W^{3/2}+C_2(a^2+b^2),
\end{equation}
for $\lambda,C_1,C_2>0$.
To compensate the last term, we augment this Lyapunov function and define $S=W+\frac{c}{2}(\frac{a^2}{d_1}+\frac{b^2}{d_2})$. Then,
\begin{equation}
\dot S\leq - \lambda_1 W+C_1 W^{3/2}+(C_2-c)(a^2+b^2),
\end{equation}
and choosing $c>C_2$, one obtains
\begin{equation}
\dot S\leq - \lambda_2 S+C_1 S^{3/2},
\end{equation}
for some positive $\lambda_2$.
Following~\cite{coron} and noting $\Vert \gamma \Vert_{\infty}+\Vert \eta \Vert_{\infty}\leq C_2 S$, then for sufficiently small $S(0)$, it follows that $S(t) \rightarrow 0$ exponentially.

Given that $W$ (by Proposition~\ref{prop-equiv}) is equivalent to the $H^2$ norm of $\gamma$ when $\Vert \gamma \Vert_{\infty}+\Vert \eta \Vert_{\infty}$ is sufficiently small, and since by construction $\gamma_0$ verifies the required second-order compatibility conditions, there exists $\delta>0$ and $c>0$ such that if $\Vert \gamma_0\Vert_{H^2} \leq \delta$, then:
\begin{equation}
\Vert \gamma \Vert_{H^2}^2+a(t)^2+b(t)^2 \leq c\,\mathrm{e}^{-\lambda t}\left( \Vert \gamma_0 \Vert_{H^2}^2+a(0)^2+b(0)^2 \right).
\end{equation}
Since, as we argued, for small enough $\Vert z \Vert_{H^2}$ the $H^2$ norms of $z$ and $\gamma$ are equivalent, this proves Theorem~\ref{thm-main}.
\begin{remark}\em\label{rem-qzero}
The proof has been carried out for the case $q\neq 0$. If $q=0$, we have to modify the target system following Section~\ref{sect-q0} and this implies the appearance of a linear boundary term (a coefficient times $\beta(0,t)$) in the $\gamma$ system; similarly, in the $\eta$ and $\theta$ systems, $\eta_2(0,t)$ and $\theta_2(0,t)$ terms will appear. These terms can be controlled using the same Lyapunov function by following the strategy outlined in Section~\ref{sect-q0}
\end{remark}

\section{Concluding remarks}\label{sect:conclusions}
We have solved the problem of full-state boundary stabilization for a $2\times2$ system of first-order hyperbolic quasilinear PDEs with actuation on only one boundary. We have shown, using a strict Lyapunov function, $H^2$ local exponential stability of the state. It is possible to extend this result to design an observer, as shown in~\cite{vazquez-obs}, and combining both results one obtains an output-feedback controller with similar properties (see~\cite{vazquez-of}).

It would be of interest to extend the method to $n\times n$ systems. For instance, a $3 \times 3$ first-order hyperbolic system of interest is the Saint-Venant-Exner system, which models open channels with a moving sediment bed~\cite{diagne}; the extension is shown (for the linear case) in~\cite{florent1}. While extending the Lyapunov analysis to $n \times n$ systems has been done~\cite{coron2}, considerable extra effort is required to extend backstepping to a general $n \times n$ system, even in the linear case. In general, the method  needs $n^2$ kernels resulting in a $n^2 \times n^2$ system of coupled first-order hyperbolic equations, whose well-posedness depends critically on the exact choice of the transformation and target system. The extension has been shown possible, for the linear case, if the system has $n$ positive and one negative transport speeds, with actuation only on the state corresponding to the negative velocity~\cite{florent2}.

\appendix

 \section{Well-posedness of the kernel equations}\label{sec:wp}
 We show well-posedness of the following hyperbolic $4\times4$ system, which is generic enough to contain all the kernel equation systems that appear in the paper:
\begin{eqnarray}\label{eqn-hypF1}
\epsilon_1(x)F^1_x+\epsilon_1(\xi)F^1_\xi&=&g_1(x,\xi)+\sum_{i=1}^4 C_{1i} (x,\xi)  F^i(x,\xi),\quad\,\\
\epsilon_1(x)F^2_x-\epsilon_2(\xi)F^2_\xi&=&g_2(x,\xi)+\sum_{i=1}^4 C_{2i}(x,\xi) F^i(x,\xi),\label{eqn-hypF2}\\
\epsilon_2(x)F^3_x-\epsilon_1(\xi)F^3_\xi&=&g_3(x,\xi)+\sum_{i=1}^4 C_{3i}(x,\xi) F^i(x,\xi),\label{eqn-hypF3}\\
\epsilon_2(x)F^4_x+\epsilon_2(\xi)F^4_\xi&=&g_4(x,\xi)+\sum_{i=1}^4 C_{4i}(x,\xi) F^i(x,\xi),\quad\,\,\,\,\,\, \label{eqn-hypF4}
\end{eqnarray}
evolving in the domain $\mathcal T=\{(x,\xi):0\leq \xi \leq x \leq 1\}$,
with boundary conditions:
\begin{eqnarray} \label{eqn-bcF1}
F^1(x,0)&=&h_1(x)+q_1(x)F^2(x,0)+q_2(x)F^3(x,0),\quad\,\,\\
F^2(x,x)&=&h_2(x),\quad
F^3(x,x)=h_3(x), \label{eqn-bcF2}\\
F^4(x,0)&=&h_4(x)+q_3(x)F^2(x,0)+q_4(x)F^3(x,0).\quad\,\,\label{eqn-bcF4}
\end{eqnarray}
This type of system has been called ``generalized Goursat problem'' by some authors~\cite{holten}. However the boundaries of the domain $\mathcal T$ are characteristic for (\ref{eqn-hypF1}) and (\ref{eqn-hypF4}), thus the general results derived in \cite{holten} cannot be applied. The following theorems discusses existence, uniqueness and smoothness of solutions to the equations.

\begin{theorem}\label{th-wp}
Consider the hyperbolic system (\ref{eqn-hypF1})--(\ref{eqn-bcF4}).
Under the assumptions
$
q_i,h_i\in \mathcal C([0,1]),\,g_i,C_{ji}\in \mathcal C(\mathcal T),\,\,i,j=1,2,3,4
$
and $\epsilon_1,\epsilon_2\in \mathcal C ([0,1])$ with $\epsilon_1(x),\epsilon_2(x)>0$, there exists a unique $\mathcal C(\mathcal T)$ solution $F^i$, $i=1,2,3,4$.
\end{theorem}

\begin{theorem}\label{th-sm}
Consider the hyperbolic system (\ref{eqn-hypF1})--(\ref{eqn-bcF4}). Under the assumptions of Theorem~\ref{th-wp}, and the additional assumptions
$
\epsilon_i,q_i,h_i\in \mathcal C^N([0,1]),\,g_i,C_{ji}\in \mathcal C^N(\mathcal T)$,
there exists a unique $\mathcal C^N(\mathcal T)$ solution $F^i$, $i=1,2,3,4$.
\end{theorem}

Next we prove the theorems; the proof is based on transforming the equations into integral equations and then solving them using a successive approximation method.

\subsection{Transformation to integral equations}\label{sect-kerneltrans}
The equations can be transformed into integral equations by the method of characteristics. For that, it is necessary to define:
\begin{equation}
\phi_1(x)=\int_0^x\frac{1}{\epsilon_1(z)}dz,\,\phi_2(x)=\int_0^x\frac{1}{\epsilon_2(z)}dz,
\end{equation}
and $\phi_3(x)=\phi_1(x)+\phi_2(x)$.
Note that all the $\phi$ functions are monotonically increasing and thus invertible, due to positivity of the $\epsilon$ coefficients. Under the assumptions of Theorem~\ref{th-wp}, it also holds that $\phi_i,\phi_i^{-1}\in\mathcal C^1([0,1])$.

Define, for $(x,\xi)\in\mathcal T$, the characteristic lines along which (\ref{eqn-hypF1})--(\ref{eqn-hypF4}) evolve:
\begin{eqnarray} \label{eqn-x1}
x_1(x,\xi,s)&=&\phi_1^{-1}\left(\phi_1(x)-\phi_1(\xi)+s\right),\,\\ \xi_1(x,\xi,s)&=&\phi_1^{-1}(s),\\
x_2(x,\xi,s)&=&\phi_1^{-1}\left(\phi_1\left(\phi_3^{-1}\left(\phi_1(x)+\phi_2(\xi)\right)\right)
+s\right), \\
\xi_2(x,\xi,s)&=&\phi_2^{-1}\left(\phi_2\left(\phi_3^{-1}\left(\phi_1(x)+\phi_2(\xi)\right)\right)
-s\right),\quad\\
x_3(x,\xi,s)&=&\phi_2^{-1}\left(\phi_2\left(\phi_3^{-1}\left(\phi_2(x)+\phi_1(\xi)\right)\right)
+s\right),\\ \xi_3(x,\xi,s)&=&\phi_1^{-1}\left(\phi_1\left(\phi_3^{-1}\left(\phi_2(x)+\phi_1(\xi)\right)\right)
-s\right),\\
x_4(x,\xi,s)&=&\phi_2^{-1}\left(\phi_2(x)-\phi_2(\xi)+s\right),\\\xi_4(x,\xi,s)&=&\phi_2^{-1}(s), \label{eqn-xi4}
\end{eqnarray}
 where the argument $s$ that parameterizes $x_i$ and $\xi_i$ belongs to the interval $[0,s_i^F]$, with $s_i^F$ defined as
\begin{eqnarray} \label{eqn-s1f}
s_1^F(x,\xi)&=&\phi_1(\xi),\\
s_2^F(x,\xi)&=&\phi_1(x)-\phi_1\left(\phi_3^{-1}\left(\phi_1(x)+\phi_2(\xi)\right)\right),\\
s_3^F(x,\xi)&=&\phi_2(x)-\phi_2\left(\phi_3^{-1}\left(\phi_2(x)+\phi_1(\xi)\right)\right),\quad\\
s_4^F(x,\xi)&=&\phi_2(\xi). \label{eqn-s4f}
 \end{eqnarray}
 The following holds
 \begin{lemma}
If $(x,\xi)\in\mathcal T$ and $s\in[0,s_i^F]$, it holds that $(x_i(x,\xi,s),\xi_i(x,\xi,s))\in\mathcal T$, for $i=1,\hdots,4$.  Also, under the assumptions of Theorem~\ref{th-wp}, $x_i$, $\xi_i$, and $s_i^F$ are continuous in their domains of definition  since they are defined as compositions of continuous functions. Moreover, the following inequalities are verified
\begin{eqnarray}
x_i(x,\xi,s)&\leq& x,\quad i=1,\hdots,4,\\
\xi_1,\xi_4(x,\xi,s) &\leq& \xi, \quad
\xi_2,\xi_3(x,\xi,s) \geq \xi.
\end{eqnarray}
 \end{lemma}

 Using these definitions,  (\ref{eqn-hypF1})--(\ref{eqn-hypF4})  are integrated to:
 \begin{eqnarray}
F^j(x,\xi)&=&F^j\left(x_j(x,\xi,0),\xi_j(x,\xi,0)\right)
+G_j(x,\xi)
+I_j[F](x,\xi),
\end{eqnarray}
where we have denoted
\begin{eqnarray}
 F&=&\left[\begin{array}{c}F^1\\F^2\\F^3\\F^4\end{array}\right],\quad
 G_j(x,\xi)=\int_0^{s^F_j(x,\xi)} g_j\left(x_j(x,\xi,s),\xi_j(x,\xi,s)\right)   ds,
\\ I_j[F](x,\xi)&=&
\sum_{i=1}^4 \int_0^{s^F_j(x,\xi)} \hspace{-10pt}
 C_{ji}\left(x_j(x,\xi,s),\xi_j(x,\xi,s)\right)
F^i\left(x_j(x,\xi,s),\xi_j(x,\xi,s)\right)
 ds,\label{eqn-integrat} \qquad
\end{eqnarray}
for $j=1,2,3,4$. Substituting the boundary conditions (\ref{eqn-bcF1})--(\ref{eqn-bcF4}) and expressing the terms in  (\ref{eqn-bcF1}) and (\ref{eqn-bcF4}) containing $F^2(x,0)$ and $F^3(x,0)$ in terms of the solution (\ref{eqn-integrat}) we get four integral equations which have the following structure:
 \begin{eqnarray}
F^j(x,\xi)&=&H_j(x,\xi)+G_j(x,\xi)+
\varphi_j(x,\xi)
+Q_j[F](x,\xi)+I_j[F](x,\xi),\qquad
\end{eqnarray}
where $H_j(x,\xi)=h_j(x_j(x,\xi,0))$, $\varphi_j(x,\xi)$ is has the values $\varphi_2=\varphi_3=0$ and
\begin{eqnarray}
\varphi_1&=&q_1(x_1(x,\xi,0))H_2(x_1(x,\xi,0),0)+q_2(x_1(x,\xi,0))H_3(x_1(x,\xi,0),0)
\nonumber \\ &&
+q_1(x_1(x,\xi,0))G_2(x_1(x,\xi,0),0)
+q_2(x_1(x,\xi,0))
G_3(x_1(x,\xi,0),0)
,\qquad \,\,\\
\varphi_4&=&q_3(x_4(x,\xi,0))H_2(x_4(x,\xi,0),0)+q_4(x_4(x,\xi,0))H_3(x_4(x,\xi,0),0)
\nonumber \\ &&
+q_3(x_4(x,\xi,0))G_2(x_4(x,\xi,0),0)
+q_4(x_4(x,\xi,0))
G_3(x_4(x,\xi,0),0)
.\qquad \,\,
\end{eqnarray}
and the values of the $Q_j[F](x,\xi)$ are $Q_2=Q_3=0$ and
\begin{eqnarray}
Q_1[F]&=&q_1(x_1(x,\xi,0))I_2[F](x_1(x,\xi,0),0)
+q_2(x_1(x,\xi,0))I_3[F](x_1(x,\xi,0),0),\qquad \,\,\\
Q_4[F]&=&q_3(x_4(x,\xi,0))I_2[F](x_4(x,\xi,0),0)
+q_4(x_4(x,\xi,0))I_3[F](x_4(x,\xi,0),0).\qquad \,\,
\end{eqnarray}

In this form, the equations are amenable to be solved using the successive approximation method. This is explained next.

\subsection{Solution of the integral equation via a successive approximation series}\label{sect-kernelapprox}

The successive approximation method can be used to solve the integral equations. Define first the following functional acting on $ F$:
\begin{equation}
\Phi_j[F](x,\xi)=Q_j[F](x,\xi)+I_j[F](x,\xi),
\end{equation}
and the vectors:
\begin{equation}
\varphi=\left[\begin{array}{c}H_1+G_1+\varphi_1\\H_2+G_2+\varphi_2\\H_3+G_3+\varphi_3\\H_4+G_4+\varphi_4\end{array}\right],\,
 \Phi[F]=\left[\begin{array}{c}\Phi_1[F]\\\Phi_2[F]\\\Phi_3[F]\\\Phi_4[F]\end{array}\right].\,
\end{equation}
Define then
\begin{eqnarray}
 F^0(x,\xi)&=&\varphi(x,\xi),
 F^{n}(x,\xi)=\Phi[ F^{n-1}](x,\xi).\quad \label{eqn-Fn}
\end{eqnarray}
Finally define for $n\geq1$ the increment $\Delta  F^n= F^{n}- F^{n-1}$, with $\Delta  F^0=\varphi$ by definition. It is easy to see that, since $\Phi$ is a linear functional, the equation $\Delta  F^n(x,\xi)=\Phi[ \Delta F^{n-1}](x,\xi)$ holds.

If $\lim_{n\rightarrow \infty} F^n(x,\xi)$ exists, then $ F=\lim_{n\rightarrow \infty} F^n(x,\xi)$ is a solution of the integral equations (and thus solves the original hyperbolic system). Using the definition of $\Delta  F^n$, it follows that  if $\sum_{n=0}^{\infty}\Delta  F^n(x,\xi)$ converges, then
\begin{equation}
 F (x,\xi)=\sum_{n=0}^{\infty}\Delta  F^n(x,\xi).\label{eqn-Fseriesdef}
\end{equation}
\subsection{Proof of convergence of the successive approximation series}
First, define:
\begin{eqnarray}
\bar \phi&=&\max_{\substack{(x,\xi)\in\mathcal T\\
i=1,2,3,4}} \left\{ \left| \varphi_i(x,\xi) \right|\right\},\,\, \bar C_{ji}=\max_{(x,\xi)\in\mathcal T}\left|C_{ji}(x,\xi)\right|,\,\, K_{\epsilon}=\max_{(x,\xi)\in\mathcal T} \left\{\frac{1}{\epsilon_1(x)},\frac{1}{\epsilon_2(x)}\right\},
\nonumber
\\ \bar q_i&=&\max_{x\in[0,1]} \left| q_i(x) \right|,
\bar C= \left(1+\sum_{i=1}^4\bar q_{i} \right)\left(\sum_{j=1}^4\sum_{i=1}^4\bar C_{ji} \right).
\end{eqnarray}
Next, we prove the following two lemmas:
\begin{lemma}\label{lem-int}
For $i=1,2,3,4$, $n\geq1$, $(x,\xi)\in\mathcal T$, and $s^F_i(x,\xi)$, $x_i(x,\xi,s)$ defined as in (\ref{eqn-x1})--(\ref{eqn-s4f}), it follows that
\begin{equation}
\int_0^{s^F_i(x,\xi)}x^n_i(x,\xi,s)ds\leq K_{\epsilon} \frac{x^{n+1}}{n+1}.
\end{equation}
\end{lemma}
\vspace{2pt}
\begin{proof}
We show the result for $i=1,2$. It follows for $i=3,4$ by switching $\epsilon_1$ and $\phi_1$, respectively, for $\epsilon_2$ and $\phi_2$.
For $i=1$ we can write:
\begin{equation}
\int_0^{\phi_1(\xi)}  x_1^n(x,\xi,s)ds=\int_0^{\phi_1(\xi)}   \left[\phi_1^{-1}\left(\phi_1(x)-\phi_1(\xi)+s\right)) \right]^n ds.
\end{equation}
To prove the inequality, change the variable of integration to $z=\phi_1^{-1}\left(\phi_1(x)-\phi_1(\xi)+s\right)$. Then, taking into account
\begin{equation}
\frac{dz}{ds}
=\frac{d}{ds}\left[\phi_1^{-1}\left(\phi_1(x)-\phi_1(\xi)+s\right)\right]=\frac{1}{\phi_1'(z)}=\epsilon_1(z),
\end{equation}
the integral can be bounded as follows:
\begin{eqnarray}
&& \int_0^{\phi_1(\xi)} \left[\phi_1^{-1}\left(\phi_1(x)-\phi_1(\xi)+s\right)) \right]^n ds
\nonumber \\
&=&
\int_{\phi_1^{-1}\left(\phi_1(x)-\phi_1(\xi)\right)}^{x}  z^n/ \epsilon_1(z) dz \leq K_{\epsilon} \int_{0}^{x} z^n dz=K_{\epsilon} \frac{x^{n+1}}{n+1}.
\end{eqnarray}

For $i=2$ the integral can be written as:
\begin{eqnarray}
 \int_0^{\phi_1(x)-\phi_1\left(\phi_3^{-1}\left(\phi_1(x)+\phi_2(\xi)\right)\right)}
 \hspace{-100pt}
 \left[\phi_1^{-1}\left(\phi_1\left(\phi_3^{-1}\left(\phi_1(x)+\phi_2(\xi)\right)\right)+s\right)\right]^n ds.
\end{eqnarray}
As before,
 change the variable of integration to $z=\phi_1^{-1}\left(\phi_1\left(\phi_3^{-1}\left(\phi_1(x)+\phi_2(\xi)\right)\right)+s\right)$. Then
 one  has that $
\frac{dz}{ds}
=\frac{1}{\phi_1'(z)}=\epsilon_1(z)$,
thus
the integral can be bounded as follows:
\begin{eqnarray}
&& \int_0^{\phi_1(x)-\phi_1\left(\phi_3^{-1}\left(\phi_1(x)+\phi_2(\xi)\right)\right)} \hspace{-100pt}  \left[\phi_1^{-1}\left(\phi_1\left(\phi_3^{-1}\left(\phi_1(x)+\phi_2(\xi)\right)\right)+s\right)\right]^n ds
\nonumber \\
&=&
\int_{\phi_3^{-1}\left(\phi_1(x)+\phi_2(\xi)\right)}^{x}  z^n/ \epsilon_1(z) dz \leq K_{\epsilon} \int_{0}^{x} z^n dz=K_{\epsilon} \frac{x^{n+1}}{n+1},\qquad\,
\end{eqnarray}
which concludes the proof.
\end{proof}
\begin{lemma}\label{lemma-bound}
For $i=1,2,3,4$, $n\geq1$ and $(x,\xi)\in\mathcal T$, assume that
\begin{equation}
\left| \Delta F^n_i (x,\xi) \right|  \leq \bar \phi \frac{\bar C^n K_{\epsilon}^n x^n}{n!},
\end{equation}
then it follows that
$
\left| \Phi_i(\Delta  F^n) (x,\xi) \right| \leq \bar \phi \frac{\bar C^{n+1} K_{\epsilon}^{n+1} x^{n+1}}{(n+1)!}$.
\end{lemma}
\begin{proof}
We show it for $i=1,2$; the structure of the equations is the same for $i=3,4$. For $i=2$:
\begin{eqnarray}
\left| \Phi_2[\Delta  F^n](x,\xi)\right|&=&\left| I_2[\Delta  F^n](x,\xi)\right|
\leq
 \sum_{i=1}^4  \bar C_{2i} \int_0^{s^F_2(x,\xi)} \left| \Delta F^n_i\left(x_2(x,\xi,s),\xi_2(x,\xi,s)
 \right) ds \right|
  \nonumber \\ &\leq &\bar \phi
\frac{K_{\epsilon}^n \bar C^{n}}{n!} \sum_{i=1}^4 \bar C_{2i} \int_0^{s^F_2(x,\xi)} x_2^n(x,\xi,s) ds
\leq \bar \phi
\frac{K_{\epsilon}^{n+1} \bar C^{n+1} x^{n+1} }{(n+1)!},
 \end{eqnarray}
 where Lemma~\ref{lem-int} has been applied. Similarly, for $i=1$:
 \begin{eqnarray}
 \Phi_1[\Delta  F^n](x,\xi) &\leq& \left| Q_1[\Delta  F^n](x,\xi)\right|+\left| I_1[\Delta  F^n](x,\xi)\right|
\nonumber \\ &\leq&
 \bar q_1
 \bar \phi \frac{\bar C^n K_{\epsilon}^n}{n!}
 \sum_{i=1}^4 \bar C_{2i} \int_0^{s^F_2(x_1(x,\xi,0),0)}
x_2^n(x_1(x,\xi,0),0,s) ds
\nonumber \\ &&
+\bar q_2 \bar \phi  \frac{\bar C^n K_{\epsilon}^n}{n!} \sum_{i=1}^4 \bar C_{3i} \int_0^{s^F_3(x_1(x,\xi,0),0)}
x_3^n(x_1(x,\xi,0),0,s) ds
\nonumber \\ &&
+\bar \phi  \frac{\bar C^n K_{\epsilon}^n}{n!}
\sum_{i=1}^4 \bar C_{1i} \int_0^{s^F_1(x,\xi)}x_1^n(x,\xi,s)ds
 \nonumber \\ &\leq&
 \bar q_1
 \bar \phi \frac{\bar C^n K_{\epsilon}^{n+1}}{n!}
 \sum_{i=1}^4 \bar C_{2i} \frac{x_1(x,\xi,0)^n}{n+1}
+\bar q_2 \bar \phi  \frac{\bar C^n K_{\epsilon}^{n+1}}{n!} \sum_{i=1}^4 \bar C_{3i} \frac{x_1(x,\xi,0)^n}{n+1}
\nonumber \\ &&
+\bar \phi  \frac{\bar C^n K_{\epsilon}^{n+1}}{n!}
\sum_{i=1}^4 \bar C_{1i} \frac{x^n}{n+1}
 \leq \bar \phi \frac{\bar C^{n+1} K_{\epsilon}^{n+1}x^{n+1}}{(n+1)!} ,
\end{eqnarray}
since $x_i(x,\xi,0)\leq x$. Thus the lemma is proved.
\end{proof}

Next we show that (\ref{eqn-Fseriesdef}) converges.
\begin{proposition}\label{prop-conv}
For $\Delta F_i^n(x,\xi)$, $i=1,2,3,4$, one has that
\begin{equation}
\left| \sum_{n=0}^{\infty} \Delta F_i^n(x,\xi) \right| \leq
\bar \phi \mathrm{e}^{\bar C K_{\epsilon} x}.
\end{equation}
\end{proposition}

\begin{proof}
The result follows if we show that
$
\left| \Delta F_i^n(x,\xi) \right| \leq
\bar \phi \frac{\bar C^{n} K_{\epsilon}^{n} x^{n}}{n!}$.
We prove the bound by induction. For $n=0$, the result follows from (\ref{eqn-Fn}). Assume that the bound is correct for all $i$ in $\Delta  F^n(x,\xi)$. Then, we get for $\Delta F_i^{n+1}(x,\xi)$ that
\begin{equation}
\left| \Delta F_i^{n+1} (x,\xi) \right|
=\left| \Phi_i [\Delta  F^{n}] (x,\xi) \right|
\leq
\bar \phi \frac{\bar C^{n+1} K_{\epsilon}^{n+1} x^{n+1}}{(n+1)!},
\end{equation}
where we have used Lemma~\ref{lemma-bound}. Thus the proposition follows.
\end{proof}

From Proposition~\ref{prop-conv} we conclude that the successive approximation series is bounded and converges uniformly. Thus, a bounded solution to Equations~(\ref{eqn-hypF1})--(\ref{eqn-bcF4}) exists. This proves the existence part of  Theorem~\ref{th-wp}.

To prove uniqueness, let us denote by $ F(x,\xi)$ and $ F'(x,\xi)$ two different solutions to (\ref{eqn-hypF1})--(\ref{eqn-bcF4}). Defining $ {\tilde F}(x,\xi)= F(x,\xi)- F'(x,\xi)$. By linearity of (\ref{eqn-hypF1})--(\ref{eqn-bcF4}), $ {\tilde F}(x,\xi)$ also verifies (\ref{eqn-hypF1})--(\ref{eqn-bcF4}), with $h_i=0$ for all $i$. Then $\bar \phi=0$ for ${\tilde F}(x,\xi)$, and Proposition~\ref{prop-conv} we conclude $ {\tilde F}(x,\xi)=0$, which implies that $ F(x,\xi)= F'(x,\xi)$.

To prove that the solution is continuous, note that since (\ref{eqn-Fseriesdef}) converges uniformly, one only needs to prove continuity of each term. First, $\Delta F_0=\varphi_i\in \mathcal C(\mathcal T)$ since the $\varphi_i$ are defined as a sum of compositions of continuous functions. Similarly, since $\Delta F_{n}$ is defined as the integral (with continuous limits) of continuous functions times the previous $\Delta F_{n-1}$ composed with continuous functions, by induction it can be shown that $\Delta F_n\in \mathcal C(\mathcal T)$. Thus $F\in \mathcal C(\mathcal T)$ and Theorem~\ref{th-wp} is proved.

\subsection{Smoothness of solutions}
Next we sketch the proof of Theorem~\ref{th-sm}. We only consider $N=1$; for $N\geq1$  the result can be proven by induction.
Denote $G_i=\frac{\partial F_i}{\partial x}(x,\xi)$ and $H_i=\frac{\partial F_i}{\partial \xi}(x,\xi)$. By  differentiating with respect to $x$ and $\xi$ in (\ref{eqn-hypF1})--(\ref{eqn-hypF4}), we find two uncoupled $4 \times 4$ hyperbolic systems  (for $G_i$ and for $H_i$)
\begin{eqnarray}\label{eqn-hypG1}
\epsilon_1(x)G^1_x+\epsilon_1(\xi)G^1_\xi&=&-\epsilon_1'(x)F^1+\frac{\partial g_1}{\partial x}(x,\xi)+\sum_{i=1}^4 \frac{\partial C_{1i}}{\partial x}(x,\xi)  F^i(x,\xi)
\nonumber \\ &&
+\sum_{i=1}^4 C_{1i} (x,\xi)  G^i(x,\xi),\qquad\\
\epsilon_1(x)G^2_x-\epsilon_2(\xi)G^2_\xi&=&-\epsilon_1'(x)F^2+\frac{\partial g_2}{\partial x}(x,\xi)+\sum_{i=1}^4 \frac{\partial C_{2i}}{\partial x}(x,\xi)F^i(x,\xi)
\nonumber \\ &&
+\sum_{i=1}^4 C_{2i}(x,\xi) G^i(x,\xi),\\
\epsilon_2(x)G^3_x-\epsilon_1(\xi)G^3_\xi&=&-\epsilon_2'(x)F^3+\frac{\partial g_3}{\partial x}(x,\xi)+\sum_{i=1}^4 \frac{\partial C_{3i}}{\partial x}(x,\xi) F^i(x,\xi)
\nonumber \\ &&
+\sum_{i=1}^4 C_{3i}(x,\xi) G^i(x,\xi),\\
\epsilon_2(x)G^4_x+\epsilon_2(\xi)G^4_\xi&=&-\epsilon_2'(x)F^4+\frac{\partial g_4}{\partial x}(x,\xi)+\sum_{i=1}^4 \frac{\partial C_{4i}}{\partial x}(x,\xi) F^i(x,\xi)
\nonumber \\ &&
+\sum_{i=1}^4 C_{4i}(x,\xi) G^i(x,\xi), \label{eqn-hypG4}
 \\
\label{eqn-hypH1}
\epsilon_1(x)H^1_x+\epsilon_1(\xi)H^1_\xi&=&-\epsilon_1'(\xi)F^1+\frac{\partial g_1}{\partial \xi}(x,\xi)+\sum_{i=1}^4 \frac{\partial C_{1i}}{\partial \xi}(x,\xi)  F^i(x,\xi)
\nonumber \\ &&
+\sum_{i=1}^4 C_{1i} (x,\xi)  H^i(x,\xi),\qquad\\
\epsilon_1(x)H^2_x-\epsilon_2(\xi)H^2_\xi&=&-\epsilon_1'(\xi)F^2+\frac{\partial g_2}{\partial \xi}(x,\xi)+\sum_{i=1}^4 \frac{\partial C_{2i}}{\partial \xi}(x,\xi)F^i(x,\xi)
\nonumber \\ &&
+\sum_{i=1}^4 C_{2i}(x,\xi) H^i(x,\xi),\\
\epsilon_2(x)H^3_x-\epsilon_1(\xi)H^3_\xi&=&-\epsilon_2'(\xi)F^3+\frac{\partial g_3}{\partial \xi}(x,\xi)+\sum_{i=1}^4 \frac{\partial C_{3i}}{\partial \xi}(x,\xi) F^i(x,\xi)
\nonumber \\ &&
+\sum_{i=1}^4 C_{3i}(x,\xi) H^i(x,\xi),\\
\epsilon_2(x)H^4_x+\epsilon_2(\xi)H^4_\xi&=&-\epsilon_2'(\xi)F^4+\frac{\partial g_4}{\partial \xi}(x,\xi)+\sum_{i=1}^4 \frac{\partial C_{4i}}{\partial \xi}(x,\xi) F^i(x,\xi)
\nonumber \\ &&
+\sum_{i=1}^4 C_{4i}(x,\xi) H^i(x,\xi). \label{eqn-hypH4}
\end{eqnarray}

Now, differentiating the boundary conditions in (\ref{eqn-bcF2}) it is found that:
\begin{eqnarray}
G^2(x,x)+H^2(x,x)&=&h_2'(x), \quad
G^3(x,x)+H^3(x,x)=h_3'(x),
\end{eqnarray}
and setting $x=\xi$ in (\ref{eqn-hypF2})--(\ref{eqn-hypF3})
\begin{eqnarray}
\epsilon_1(x)G^2(x,x)-\epsilon_2(x)H^2(x,x)&=&g_2(x,x)+\sum_{i=1}^4 C_{2i}(x,x) F^i(x,x),\\
\epsilon_2(x)G^3(x,x)-\epsilon_1(\xi)H^3(x,x)&=&g_3(x,x)+\sum_{i=1}^4 C_{3i}(x,x) F^i(x,x),
\end{eqnarray}
we can find a set of boundary conditions for $G^j$ and $H^j$, $j=2,3$, at the boundary $x=\xi$:
\begin{eqnarray} \label{eqn-bcG1}
G^2(x,x)&=&\frac{\epsilon_2(x)h_2'(x)+ g_2(x,x)+\sum_{i=1}^4 C_{2i} (x,x)  F^i(x,x)}{\epsilon_2(x)+\epsilon_1(x)},\\
G^3(x,x)&=&\frac{\epsilon_1(x) h_3'(x)+g_3(x,x)+\sum_{i=1}^4 C_{3i} (x,x)  F^i(x,x)}{\epsilon_2(x)+\epsilon_1(x)},\\
H^2(x,x)&=&\frac{\epsilon_1(x)h_2'(x)- g_2(x,x)-\sum_{i=1}^4 C_{2i} (x,x)  F^i(x,x)}{\epsilon_2(x)+\epsilon_1(x)},\\
H^3(x,x)&=&\frac{\epsilon_2(x) h_3'(x)- g_3(x,x)-\sum_{i=1}^4 C_{3i} (x,x)  F^i(x,x)}{\epsilon_2(x)+\epsilon_1(x)},
\end{eqnarray}

Similarly, differentiating  boundary conditions (\ref{eqn-bcF1}) and (\ref{eqn-bcF4}) we find boundary conditions for $G^1$ and $G^4$ at $\xi=0$:
\begin{eqnarray} \label{eqn-bcGH1}
G^1(x,0)&=&h_1'(x)+q_1'(x)F^2(x,0)+q_2'(x)F^3(x,0)+q_1(x)G^2(x,0)
\nonumber \\&&
+q_2(x)G^3(x,0),\\
G^4(x,0)&=&h_4'(x)+q_3'(x)F^2(x,0)+q_4'(x)F^3(x,0)+q_3(x)G^2(x,0)\nonumber \\&&+q_4(x)G^3(x,0),\label{eqn-bcGh4}
\end{eqnarray}
and setting $\xi=0$ in (\ref{eqn-hypF1})--(\ref{eqn-hypF4})
we can also find two sets of boundary conditions for $H^j$, $j=1,4$, at the boundary $\xi=0$:\begin{eqnarray} \label{eqn-bcH1}
H^1(x,0)&=&\frac{g_1(x,0)+\sum_{i=1}^4 C_{1i} (x,0)  F^i(x,0)}{\epsilon_1(0)}-\frac{\epsilon_1(x)}{\epsilon_1(0)} \Bigg(h_1'(x)+q_1'(x)F^2(x,0)
 \nonumber \\ && \left. +q_2'(x)F^3(x,0)
+\frac{q_1(x)}{\epsilon_1(x)}\left(\epsilon_2(0)H^2(x,0)+g_2(x,0)+\sum_{i=1}^4 C_{2i} (x,0)  F^i(x,0)\right)
\right. \nonumber \\ && \left.
+\frac{q_2(x)}{\epsilon_2(x)}\left(\epsilon_1(0)H^3(x,0)+g_3(x,0)+\sum_{i=1}^4 C_{3i} (x,0)  F^i(x,0)\right)\right),\\
H^4(x,0)&=&\frac{g_4(x,0)+\sum_{i=1}^4 C_{4i} (x,0)  F^i(x,0)}{\epsilon_2(0)}-\frac{\epsilon_2(x)}{\epsilon_2(0)} \Bigg(h_4'(x)+q_3'(x)F^2(x,0)
 \nonumber \\ && \left. +q_4'(x)F^3(x,0)
+\frac{q_3(x)}{\epsilon_1(x)}\left(\epsilon_2(0)H^2(x,0)+g_2(x,0)+\sum_{i=1}^4 C_{2i} (x,0)  F^i(x,0)\right)
\right. \nonumber \\ && \left.
+\frac{q_4(x)}{\epsilon_2(x)}\left(\epsilon_1(0)H^3(x,0)+g_3(x,0)+\sum_{i=1}^4 C_{3i} (x,0)  F^i(x,0)\right)\right).\label{eqn-bcH4}
\end{eqnarray}
Thus, both the $G^i$'s and $H^i$'s verify equations formally equivalent to (\ref{eqn-hypF1})--(\ref{eqn-bcF4}), with derivatives of the old equations' coefficients as new coefficients, and the $F^i$'s as additional terms. If these equations have solutions, then the solutions must be the partial derivatives of the $F^i$ functions.

Now, under the assumptions of Theorem~\ref{th-sm}, by Theorem~\ref{th-wp} there is a (at least) continuous solution $F^i(x,\xi)$. Plugging that solution into the equations we just derived for the $G^i$'s and $H^i$'s, one obtains equations whose coefficients and boundary conditions are (at least) continuous. Hence Theorem~\ref{th-wp} can be applied implying that the $G^i$'s and $H^i$'s are continuous. Thus $F^i(x,\xi)\in\mathcal C^1\left(\mathcal T\right)$, proving the result.

\section{Technical results}

Next we give some technical lemmas used throughout the paper. The first lemma follows from the fact that the control direct and inverse kernels are $\mathcal C^2 (\mathcal T)$ functions.
\begin{lemma}
\begin{eqnarray}
\vert \mathcal K[\gamma] \vert &\leq& C_1 \left(
\vert \gamma\vert+\Vert \gamma \Vert_{L^1}
 \right),\\
\vert \mathcal L[\gamma]\vert
&\leq& C_2 \left(
\vert \gamma\vert+\Vert \gamma \Vert_{L^1}
 \right),\\
\vert \mathcal K_{1}[\gamma]\vert
&\leq& C_3 \left(
\vert \gamma\vert+\Vert \gamma \Vert_{L^1}
 \right),\\
 \vert \mathcal K_{2}[\gamma]\vert
&\leq& C_4 \left(
\vert \gamma\vert+\Vert \gamma \Vert_{L^1}
 \right),\\
\vert  \mathcal L_{1}[\gamma]\vert
&\leq& C_5 \left(
\vert \gamma\vert+\Vert \gamma \Vert_{L^1}
 \right),\\
\vert  \mathcal L_{11}[\gamma]\vert
&\leq& C_6 \left(
\vert \gamma\vert+\Vert \gamma \Vert_{L^1}
 \right).
\end{eqnarray}
\end{lemma}
The next lemma is based on the fact that, since $\Lambda_{NL}(u,x)$ is twice differentiable with respect to $u$ and $x$, and since we have $\Lambda_{NL}(0,x)=0$, it follows that there exists a  $\delta_{\Lambda}$ and $K_1$, $K_2$, $K_3$ such that if $\vert u \vert \leq \delta_{\Lambda}$, then, for any $v,w \in \mathbb R^2$, it holds that
\begin{eqnarray}
\vert \Lambda_{NL}(u,x) \vert
+\left| \frac{\partial \Lambda_{NL}(u,x)}{\partial x} \right|
&\leq& K_1 \vert u \vert ,
\\
\left| \frac{\partial \Lambda_{NL}}{\partial u} (u,x) v \right|
+\left| \frac{\partial \Lambda_{NL}(u,x)}{\partial u \partial x}v \right|
&\leq& K_2 \vert v \vert ,
\\
\left| \frac{\partial^2 \Lambda_{NL}}{\partial u^2} (u,x) v w \right|&\leq& K_3 \vert v \vert \vert w \vert.\,\,\label{eqn-linearbound}
\end{eqnarray}
Similarly, since $f_{NL}(u,x)$ is twice differentiable with respect to $u$ and once with respect to $x$, and
$f_{NL}(0,x)=\frac{\partial f_{NL}}{\partial u}(0,x)=0$, there exists a  $\delta_f$ and $K_4$, $K_5$, $K_6$ such that if $\vert u \vert \leq \delta_f $, then for any $v\in \mathbb R^2$,
\begin{eqnarray}
\vert f_{NL}(u,x) \vert+\left| \frac{\partial f_{NL}}{\partial x} (u,x) \right| &\leq& K_4 \vert u \vert^2,\\
\left| \frac{\partial f_{NL}}{\partial u} (u,x) \right| &\leq& K_5 \vert u \vert,
\\
\left| \frac{\partial^2 f_{NL}}{\partial u^2} (u,x)v \right| &\leq& K_4 \vert v \vert.\label{eqn-quadbound}
\end{eqnarray}
Then, the following lemma holds.
\begin{lemma}\label{lem-F3F4bound}
For $\Vert \gamma\Vert_{\infty}<\min\{\delta_\Lambda,\delta_f\}$,
\begin{eqnarray}
 \vert F_1 \vert &\leq &C_5 \left(
\vert \gamma\vert+\Vert \gamma \Vert_{L^1}
 \right)\label{eqn-F1bound},\\
\vert F_2 \vert &\leq &C_6 \left(
\vert \gamma\vert^2+\Vert \gamma \Vert_{L^1}^2
 \right)\label{eqn-F2bound},\\
 \vert F_3 \vert &\leq &C_7 \left(\Vert \gamma \Vert_{L^2}\Vert+\vert \gamma\vert\right)
  \left(\Vert \gamma_x \Vert_{L^2}+ \vert \gamma_x(x)
 \right),\\
 \vert F_4\vert &\leq &C_8 \left(
\vert \gamma\vert^2+\Vert \gamma \Vert_{L^2}^2
 \right).
\end{eqnarray}
\end{lemma}
The next lemma follows from the previous ones.
\begin{lemma}\label{lem-F5F6bound}
For $\Vert \gamma\Vert_{\infty}<\min\{\delta_\Lambda,\delta_f\}$,
\begin{eqnarray}
 \vert F_{11}\vert &\leq &C_9 \left(
\vert \eta\vert+\Vert \eta \Vert_{L^1}
 \right),\\
\vert F_{12} \vert &\leq &C_{10} \left(
\vert \gamma_x\vert+
\vert \gamma\vert+\Vert \gamma \Vert_{L^1}
 \right),\\
 \vert F_{21} \vert &\leq &C_{11} \left(\vert \gamma\vert+\Vert \gamma \Vert_{L^1}\right)
  \left(\vert \eta\vert+\Vert \eta \Vert_{L^1}
 \right),\\
 \vert F_5 \vert &\leq &C_{12}  \left(
\vert \eta\vert+\Vert \eta \Vert_{L^2}
 \right)
 \left(
\vert \gamma\vert+\Vert \gamma \Vert_{L^2}
 \right)
+C_{14}  \left(
\vert \eta\vert+\Vert \eta \Vert_{L^2}
 \right)
 \left(
\vert \gamma_x\vert+\Vert \gamma_x \Vert_{L^2}
 \right)
 \nonumber \\ &&
 +C_{15} \vert \gamma(0) \vert \vert \eta(0) \vert
 ,\\
  \vert F_6 \vert &\leq &C_{16}
  \left(
\vert \eta\vert+\Vert \eta \Vert_{L^2}
 \right)
 \left(
\vert \gamma\vert+\Vert \gamma \Vert_{L^2}
 \right).
  \end{eqnarray}
\end{lemma}

The next lemma follows immediately from the previous lemmas and the corresponding definitions.

\begin{lemma}\label{lem-F7F8bound}
For $\Vert \gamma\Vert_{\infty}<\min\{\delta_\Lambda,\delta_f\}$,
\begin{eqnarray}
 \vert F_{13} \vert &\leq &C_{17} \left(
\vert \eta\vert^2+\Vert \eta \Vert_{L^1}^2
 \right) + C_{18} \left(
\vert \theta\vert+\Vert \theta \Vert_{L^1}
 \right)
 ,\quad\,\\
\vert F_{14}\vert &\leq &C_{19}
 \left(
\vert \eta\vert+\Vert \eta \Vert_{L^1}
 \right)
\left(1+
\vert \gamma_x\vert+
\vert \gamma\vert+\Vert \gamma \Vert_{L^1}
 \right)
  +C_{20}\left(
\vert \eta_x\vert+
\vert \eta\vert+\Vert \eta \Vert_{L^1}
 \right)
 ,\qquad \\
 \vert F_{22}\vert &\leq &C_{21} \left(\vert \gamma\vert+\Vert \gamma \Vert_{L^1}\right)
  \left(\vert \theta\vert+\Vert \theta \Vert_{L^1}
 \right)
 +C_{22}\left(
\vert \eta\vert^2+\Vert \eta \Vert_{L^1}^2
 \right)
 ,\\
 \vert F_7 \vert &\leq &C_{23}
  \left(
\vert \eta\vert^2+\Vert \eta \Vert_{L^2}^2
 \right)
 \left(1+\Vert \gamma \Vert_{\infty}+\Vert \gamma_x \Vert_{\infty} \right)
  +C_{24}
 \left(
\vert \eta\vert+\Vert \eta \Vert_{L^2}
 \right)
    \left(
\vert \eta_x\vert+\Vert \eta \Vert_{L^2}
 \right)
  \nonumber \\ &&
  +C_{25}
 \left(
\vert \gamma\vert+\Vert \gamma \Vert_{L^2}
+\Vert \gamma_x \Vert_{\infty}
 \right)
    \left(
\vert \theta\vert+\Vert \theta \Vert_{L^2}
 \right)
 +C_{26} \left(\vert \eta(0) \vert^2+\vert \gamma(0) \vert\vert \theta(0) \vert \right)
 ,\qquad \,\,\,\\
  \vert F_8\vert &\leq &C_{27}
  \left(
\vert \eta\vert^2+\Vert \eta \Vert_{L^2}^2
 \right)\left(1+\Vert \gamma \Vert_{\infty} \right)
  +C_{28}
 \left(
\vert \gamma\vert+\Vert \gamma \Vert_{L^2}
 \right)
    \left(
\vert \theta\vert+\Vert \theta \Vert_{L^2}
 \right).
  \end{eqnarray}
\end{lemma}

Finally, the following result is crucial in establishing Theorem~\ref{thm-main}.

\begin{proposition}\label{prop-equiv}
There exists $\delta$ such that, if $\Vert \gamma \Vert_{\infty}+\Vert \eta \Vert_{\infty} < \delta$, then the following inequalities hold
\begin{eqnarray}
 \Vert \theta \Vert_{\infty} &\leq& c_1 \left(\Vert \gamma_{xx} \Vert_{\infty} +\Vert \gamma_x \Vert_{\infty}
 +\Vert \gamma \Vert_{\infty}
 \right),\\
 \Vert \theta \Vert_{L^2} &\leq& c_2\left(\Vert \gamma_{xx} \Vert_{L^2}+\Vert \gamma_x \Vert_{L^2}+\Vert \gamma \Vert_{L^2} \right),\\
 \Vert \gamma_{xx} \Vert_{\infty} &\leq& c_3\left( \Vert \theta \Vert_{\infty} +\Vert \eta \Vert_{\infty}
 +\Vert \gamma \Vert_{\infty}
 \right),\\
\Vert \gamma_{xx} \Vert_{L^2} &\leq& c_4\left( \Vert \theta \Vert_{L^2} +\Vert \eta \Vert_{L^2}
+\Vert \gamma \Vert_{L^2}
\right),
\end{eqnarray}
where  $c_1,c_2,c_3,c_4$  are positive constants.
\end{proposition}

The proposition is proven using a series three lemmas.

The first lemma gives a relation between the $L^2$ and infinity norms of $\eta$ and $\gamma_x$, under the assumption that $\Vert \gamma \Vert_{\infty} $ is small enough.
\begin{lemma}\label{lem-gammaetaequiv}
There exists $\delta_2$ such that, if $\Vert \gamma \Vert_{\infty} < \delta_2$, then the following inequalities hold
\begin{eqnarray}
 \Vert \eta \Vert_{\infty} &\leq& c_1 \left(\Vert \gamma_x \Vert_{\infty} +\Vert \gamma \Vert_{\infty} \right),\\
 \Vert \eta \Vert_{L^2} &\leq& c_2\left(\Vert \gamma_x \Vert_{L^2}+\Vert \gamma \Vert_{L^2}\right),\\
 \Vert \gamma_x \Vert_{\infty} &\leq& c_3\left( \Vert \eta \Vert_{\infty} +\Vert \gamma \Vert_{\infty} \right),\\
\Vert \gamma_x \Vert_{L^2} &\leq& c_4\left( \Vert \eta \Vert_{L^2} +\Vert \gamma \Vert_{L^2}\right),
\end{eqnarray}
where  $c_1,c_2,c_3,c_4$  are positive constants.
\end{lemma}
\begin{proof}
First, from (\ref{eqn-gammaNL2}) we see that
\begin{equation}
\eta-\Sigma(x)  \gamma_x+F_3[\gamma,\gamma_x](x)+F_4[\gamma](x)=0.
\end{equation}
Therefore, calling $\delta_1$ the value of $\delta$ in Lemma~\ref{lem-F3F4bound} and assuming $\Vert \gamma \Vert_{\infty} < \delta_1$, we can compute a bound on $ \Vert \eta \Vert_{\infty} $ as follows:
\begin{eqnarray}
 \Vert \eta \Vert_{\infty}
 &\leq& K_1  \Vert \gamma_x \Vert_{\infty}
 +\Vert F_3[\gamma,\gamma_x]\Vert_{\infty}+\Vert F_4[\gamma]\Vert_{\infty}
 \nonumber \\
 &\leq&
 K_2 \left(  \Vert \gamma_x \Vert_{\infty}
 +\Vert \gamma_x \Vert_{\infty} \Vert \gamma \Vert_{\infty}
 +\Vert \gamma \Vert_{\infty}^2 \right)
\leq
 K_3 \left(  \Vert \gamma_x \Vert_{\infty}
 +\Vert \gamma \Vert_{\infty} \right).
\end{eqnarray}
Proceeding similarly with the $L^2$ norm,
\begin{eqnarray}
 \Vert \eta \Vert_{L^2}
 &\leq& K_1  \Vert \gamma_x \Vert_{L^2}
 +\Vert F_3[\gamma,\gamma_x]\Vert_{L^2}+\Vert F_4[\gamma]\Vert_{L^2}
 \nonumber \\
 &\leq&
 K_2 \left(  \Vert \gamma_x \Vert_{L^2}
 +\Vert \gamma_x \Vert_{L^2} \Vert \gamma \Vert_{\infty}
 +\Vert \gamma \Vert_{\infty} \Vert \gamma \Vert_{L^2} \right)
\leq
 K_3 \left(  \Vert \gamma_x \Vert_{L^2}
 +\Vert \gamma \Vert_{L^2} \right).
\end{eqnarray}
For the last two inequalities, we solve for $\gamma_x$:
\begin{equation}
\gamma_x =\Sigma^{-1}(x)\left( \eta+F_3[\gamma,\gamma_x](x)+F_4[\gamma](x)\right).
\end{equation}
Remembering the definition of $\Sigma(x)$, $\bar \epsilon= \max_{x\in[0,1]}\left\{\frac{1}{\epsilon_1(x)}, \frac{1}{\epsilon_2(x)} \right\}>0$, and assuming that  $\Vert \gamma \Vert_{\infty} < \delta_1$ we obtain\begin{eqnarray}
 \Vert \gamma_x  \Vert_{\infty}
 &\leq&\bar \epsilon \left(  \Vert \eta \Vert_{\infty}
 +K_1 \Vert \gamma_x \Vert_{\infty} \Vert \gamma \Vert_{\infty}
 +K_2 \Vert \gamma \Vert_{\infty}^2 \right).\,\qquad
\end{eqnarray}
Therefore if we choose $\Vert \gamma \Vert_{\infty} < \min\left\{\delta_1,\frac{1}{2K_1\bar\epsilon}\right\}$, we reach the third inequality.
Proceeding similarly with the $L^2$ norm:
\begin{equation}
 \Vert \gamma_x  \Vert_{L^2}
\leq \bar \epsilon \left(  \Vert \eta \Vert_{L^2}
 +K_3 \Vert \gamma_x \Vert_{L^2} \Vert \gamma \Vert_{\infty}
 +K_4  \Vert \gamma \Vert_{L^2} \Vert \gamma \Vert_{\infty} \right),\quad
\end{equation}
so choosing $\Vert \gamma \Vert_{\infty} < \min\left\{\delta_1,\frac{1}{2K_3\bar\epsilon}\right\}$, we reach the fourth inequality. Therefore, choosing $\delta=\min\left\{\delta_1,\frac{1}{2K_1\bar\epsilon},\frac{1}{2K_3\bar\epsilon}\right\}$, all inequalities are verified and the lemma is proven.
\end{proof}

The next lemma gives the relation between $\eta_x$ and $\gamma_{xx}$, both in the infinity norm and the $L^2$ norm, for small $\Vert \gamma \Vert_{\infty}$:
\begin{lemma}\label{lem-gammaxxetaxequiv}
There exists $\delta$ such that, if $\Vert \gamma \Vert_{\infty} < \delta$, then the following inequalities hold
\begin{eqnarray}
 \Vert \gamma_{xx} \Vert_{\infty} &\leq& c_1 \left(\Vert \eta_x \Vert_{\infty} +\Vert \eta \Vert_{\infty}
 +\Vert \gamma \Vert_{\infty}
 \right),\\
 \Vert  \gamma_{xx} \Vert_{L^2} &\leq& c_2\left(\Vert \eta_x \Vert_{L^2}+\Vert \eta \Vert_{L^2}+\Vert \gamma \Vert_{L^2} \right),\\
 \Vert \eta_x \Vert_{\infty} &\leq& c_3\left( \Vert  \gamma_{xx} \Vert_{\infty} +\Vert \eta \Vert_{\infty}
 +\Vert \gamma \Vert_{\infty}
 \right),\\
\Vert \eta_x \Vert_{L^2} &\leq& c_4\left( \Vert  \gamma_{xx} \Vert_{L^2} +\Vert \eta \Vert_{L^2}
+\Vert \gamma \Vert_{L^2}
\right),
\end{eqnarray}
where  $c_1,c_2,c_3,c_4$  are positive constants.
\end{lemma}
\begin{proof}
Taking an $x$-derivative in (\ref{eqn-gammaNL2}):
\begin{equation}
\eta_x-\Sigma'(x)  \gamma_x-\Sigma(x)  \gamma_{xx}+F_1[\gamma]  \gamma_{xx}+
F_{32}[\gamma,\gamma_x]+F_{42}[\gamma,\gamma_x]=0,
\end{equation}
where $F_{32}$ and $F_{42}$ are defined as:
\begin{eqnarray}
F_{32}&=&\mathcal K_{2} \left[ F_1[\gamma]\gamma_x\right]
+F_{12}[\gamma,\gamma_x]\gamma_x,\\
F_{42}&=&\mathcal K_{2} \left[  F_1[\gamma] \mathcal L_{1} \left[\gamma \right] +F_2[\gamma] \right]
+F_{12}[\gamma,\gamma_x]\mathcal L_{1} \left[\gamma \right]
\nonumber \\ &&
+F_1[\gamma] \mathcal L_{11} \left[\gamma \right]+F_1[\gamma] L(x,x) \gamma_x+F_{23}[\gamma,\gamma_x] \gamma_x
,
\end{eqnarray}
where
\begin{equation}
F_{23}[\gamma,\gamma_x] =\frac{\partial f_{NL}}{\partial x}\left(\mathcal L[\gamma],x\right)+\frac{\partial f_{NL}}{\partial \gamma}\left(\mathcal L[\gamma],x\right)\left(\gamma_x+\mathcal L_{1}[\gamma]\right).
\end{equation}

These functionals verify the following bound, similar to the bounds developed in Lemma~\ref{lem-F3F4bound}, if $\Vert \gamma\Vert_{\infty}<\min\{\delta_\Lambda,\delta_f\}$:
\begin{eqnarray}
 \vert F_{32} \vert &\leq &C_1 \left(\Vert \gamma \Vert_{L^2}+\vert \gamma\vert\right)
  \left(\Vert \gamma_x \Vert_{L^2}+ \vert \gamma_x\vert
 \right)
 +C_2 \vert \gamma_x\vert^2,\\
  \vert F_{42} \vert &\leq &C_4 \left(\Vert \gamma \Vert_{L^2}+\vert \gamma\vert\right)
  \left(\Vert \gamma_x \Vert_{L^2}+ \vert \gamma_x\vert
 \right).\qquad
\end{eqnarray}
Therefore, using Lemma~\ref{lem-gammaetaequiv}, and inequality (\ref{eqn-F1bound}), and making $\Vert \gamma \Vert_{\infty}$ small enough, we can compute the bounds as in the proof of Lemma~\ref{lem-gammaetaequiv}.
\end{proof}

Finally, the next lemma relates $\eta_x$ and $\theta$, both in the infinity norm and the $L^2$ norm, for small $\Vert \gamma \Vert_{\infty}$ and $\Vert \eta \Vert_{\infty}$:
\begin{lemma}\label{lem-thetaetaequiv}
There exists $\delta$ such that, if $\Vert \gamma \Vert_{\infty}+\Vert \eta \Vert_{\infty} < \delta$, then the following inequalities hold
\begin{eqnarray}
 \Vert \theta \Vert_{\infty} &\leq& c_1 \left(\Vert \eta_x \Vert_{\infty} +\Vert \eta \Vert_{\infty}
 +\Vert \gamma \Vert_{\infty}
 \right),\\
 \Vert \theta \Vert_{L^2} &\leq& c_2\left(\Vert \eta_x \Vert_{L^2}+\Vert \eta \Vert_{L^2}+\Vert \gamma \Vert_{L^2} \right),\\
 \Vert \eta_x \Vert_{\infty} &\leq& c_3\left( \Vert \theta \Vert_{\infty} +\Vert \eta \Vert_{\infty}
 +\Vert \gamma \Vert_{\infty}
 \right),\\
\Vert \eta_x \Vert_{L^2} &\leq& c_4\left( \Vert \theta \Vert_{L^2} +\Vert \eta \Vert_{L^2}
+\Vert \gamma \Vert_{L^2}
\right),
\end{eqnarray}
where  $c_1,c_2,c_3,c_4$  are positive constants.
\end{lemma}
\begin{proof}
We can write (\ref{eqn-etaNL2}) analogously to (\ref{eqn-gammaNL2}):
\begin{equation}
\eta_t-\Sigma(x)  \eta_x+F_{31}[\gamma,\gamma_x,\eta,\eta_x](x)+F_6[\gamma,\eta](x)=0,
\end{equation}
where $F_{31}$ is defined as:
\begin{eqnarray}
F_{31}&=&\mathcal K\left[ F_1[\gamma]\eta_x
+F_{11}[\eta]\gamma_x
\right].
\end{eqnarray}
The functional $F_{31}$ verifies the following bound, similar to the bounds developed in Lemma~\ref{lem-F3F4bound}, if $\Vert \gamma\Vert_{\infty}<\min\{\delta_\Lambda,\delta_f\}$:
\begin{eqnarray}
 \vert F_{31} \vert &\leq &C_1 \left(\Vert \gamma \Vert_{L^2}+\vert \gamma\vert\right)
  \left(\Vert \eta_x \Vert_{L^2}+ \vert \eta_x\vert
 \right)
 +C_2 \left(\Vert \gamma_x \Vert_{L^2}+\vert \gamma_x\vert\right)
  \left(\Vert \eta \Vert_{L^2}+ \vert \eta\vert
 \right).\qquad
\end{eqnarray}
Therefore, using Lemma~\ref{lem-F5F6bound}, Lemma~\ref{lem-gammaetaequiv}, and inequality (\ref{eqn-F1bound}), and making $\Vert \gamma \Vert_{\infty}+\Vert \eta \Vert_{\infty}$ small enough, we can compute the bounds as in the proof of Lemma~\ref{lem-gammaetaequiv}.
\end{proof}

Combining the three lemmas, the proposition immediately follows.


\begin{thebibliography}{10}

\bibitem{bastin}
{\sc G.~Bastin and J.-M.~Coron}, {\em On boundary feedback stabilization of non-uniform linear $2 \times 2$ hyperbolic systems over a bounded interval}, Systems and Control Letters,
 60 (2011), pp.~900--906.

\bibitem{coron}
{\sc J.-M.~Coron, B.~d'Andrea-Novel, and G.~Bastin}, {\em A strict Lyapunov function for boundary control of hyperbolic systems of conservation laws}, IEEE Trans. on Automatic Control, 52 (2006), pp.~2--11.

\bibitem{coron2}
{\sc J.-M.~Coron, G.~Bastin, and B.~d'Andrea-Novel}, {\em Dissipative boundary conditions for one-dimensional nonlinear hyperbolic systems}, SIAM Journal of Control and Optimization, 47 (2008), pp.~1460--1498.

\bibitem{coron3}
{\sc J.-M.~Coron}, {\em On the null asymptotic stabilization of the two-dimensional incompressible Euler equations in a simply connected domain}, SIAM Journal of Control and Optimization, 37 (1999), pp.~1874--1896

\bibitem{culick}
{\sc F. E.C.~Culick}, {\em Nonlinear behavior of acoustic waves in combustion
chambers-I}, Acta Astronautica, 3 (1976), pp.~715--734.

\bibitem{curro}
{\sc C.~Curro, D.~Fusco and N.~Manganaro},
{\em A reduction procedure for generalized Riemann problems with application to nonlinear
transmission lines},  J. Phys. A: Math. Theor., 44 (2011), 335205.

\bibitem{diagne}
{\sc A.~Diagne, G.~Bastin, and J.-M.~Coron}, {\em Lyapunov exponential stability of linear hyperbolic systems of balance laws}, in Proceedings of the 18th IFAC World Congress, 2011.

\bibitem{dick}
{\sc M.~Dick, M.~Gugat, and G. Leugering}, {\em Classical solutions and feedback stabilisation for the gas flow in a sequence of pipes}, Networks and heterogeneous media, 5 (2010), pp.~691--709.

\bibitem{florent1}
{\sc F.~Di~Meglio, M.~Krstic, R.~Vazquez, and N.~Petit}, {\em Backstepping stabilization of an underactuated 3 X 3 linear hyperbolic system of fluid
flow transport equations},  to appear in proceedings of the 2012 ACC.

\bibitem{florent2}
{\sc F.~Di~Meglio, R.~Vazquez, and M.~Krstic}, {\em Stabilization of a hyperbolic system with one boundary controlled transport PDE coupled with $n$ counterconvecting PDEs},  submitted to 2012 CDC.


\bibitem{dos-santos}
{\sc V.~Dos Santos and C.~Prieur}, {\em Boundary control of open channels with numerical and experimental
validations}, IEEE Trans. Control Syst. Tech., 16 (2008), pp.~1252--1264.

\bibitem{goatin}
{\sc P.~Goatin}, {\em The Aw-Rascle vehicular traffic flow model with phase transitions}, Math. Comput. Modeling, 44 (2006), pp.~287--303.

\bibitem{li1}
{\sc J.-M.~Greenberg and T.-t.~Li}, {\em The effect of boundary damping for the quasilinear wave equations}, Journal of Differential Equations, 52 (1984), pp.~66--75.

\bibitem{gugat-2}
{\sc M.~Gugat and M.~Dick}, {\em Time-delayed boundary feedback stabilization of the isothermal Euler equations with friction}, Mathematical Control and Related Fields, 1 (2011), pp.469--491.

\bibitem{gugat}
{\sc M.~Gugat and M.~Herty}, {\em Existence of classical solutions and feedback stabilisation for the flow in gas networks}, ESAIM Control Optimisation and Calculus of Variations, 17 (2011), pp.~28--51.

\bibitem{GugatLeugering}  M. Gugat and G. Leugering, Global boundary controllability of the de St. Venant equations between steady states, Ann. Inst. H. Poincar\'{e} Anal. Non Lin\'{e}aire, 20 (2003),  pp.~1--11.

\bibitem{2004-Gugat-Leugering-Schmidt}
M. Gugat, G. Leugering and E. Schmidt, {\em Global
  controllability between steady supercritical flows in channel networks},
  Math. Methods Appl. Sci. 27 (2004), pp.~781--802.

  \bibitem{Halleux}
J. de~Halleux, C. Prieur, J.-M. Coron,
  B. d'Andr\'{e}a-Novel and
  G. Bastin, {\em Boundary feedback control in networks of open
  channels}, Automatica, 39 (2003), pp.~1365--1376.

\bibitem{holten}
{\sc R.P.~Holten}, {\em Generalized Goursat problem},  Pacific Journal of Mathematics, 12 (1962), pp.~207--224.


\bibitem{krstic}
{\sc M.~Krstic and A.~Smyshlyaev}, {\em Boundary Control of PDEs}, SIAM, 2008

\bibitem{krstic5}
{\sc M.~Krstic}, {\em Delay Compensation for Nonlinear, Adaptive, and PDE Systems}, Birkhauser, 2009.


\bibitem{krstic-combustion}
{\sc M.~Krstic, A.~Krupadanam, and C.~Jacobson}, {\em Self-Tuning Control of a Nonlinear Model of Combustion Instabilities},
IEEE Transactions on Control Systems Technology, 7 (1999), pp.~424--435


\bibitem{krstic3}
{\sc M.~Krstic and A.~Smyshlyaev}, {\em Backstepping boundary control for first order hyperbolic PDEs and application to systems with actuator and sensor delays}, Syst. Contr. Lett., 57 (2008), pp.~750--758.


\bibitem{li2}
{\sc T.-t.~Li}, {\em Global Classical Solutions for Quasilinear Hyperbolic Systems}, Wiley, 1994.

\bibitem{litrico}
{\sc X.~Litrico and V.~Fromion}, {\em Boundary control of hyperbolic conservation laws using a frequency domain approach},  Automatica, 45 (2009), pp.~647--656.

\bibitem{prieur}
{\sc C.~Prieur},  {\em Control of systems of conservation laws with boundary errors}, Networks and Heterogeneous Media, 4 (2009), pp.~393--407.

\bibitem{prieur2}
{\sc C.~Prieur, J.~Winkin, and G.~Bastin}, {\em Robust boundary control of systems of conservation laws}, Mathematics of Control, Signals, and Systems, 20 (2008), pp.~173--197.

\bibitem{russell}
{\sc D.L.~Russell}, {\em    Controllability and stabilizability theory for linear partial differential equations: recent progress and open questions}, SIAM Review, 20(1978), pp.~639--739.

\bibitem{hpde}
{\sc L.F.~Shampine}, {\em Solving Hyperbolic PDEs in MATLAB}, Applied Numerical Analysis \& Computational Mathematics, 2(2005), pp.~346--358.

\bibitem{smysh}
{\sc A.~Smyshlyaev and M.~Krstic}, {\em Closed form boundary state feedbacks for a class of partial integro-differential equations},  IEEE Transactions on Automatic Control, 49 (2004), pp.~2185--2202.

\bibitem{krstic2}
{\sc A.~Smyshlyaev, E.~Cerpa, and M.~Krstic}, {\em Boundary stabilization of a 1-D wave equation with in-domain antidamping}, SIAM Journal of Control and Optimization, 48 (2010), pp.~4014--4031.

\bibitem{krstic4}
{\sc A.~Smyshlyaev and M.~Krstic}, {\em Adaptive Control of Parabolic PDEs}, Princeton University Press, 2010.

 \bibitem{vazquez-coron}
{\sc R.~Vazquez, E.~Trelat and J.-M.~Coron}, {\em Control for fast and stable laminar-to-high-Reynolds-numbers transfer in a 2D Navier-Stokes channel flow},  Discrete and Continuous Dynamical Systems Series B, 10(2008), pp.~925--956.


\bibitem{vazquez}
{\sc R.~Vazquez and M.~Krstic}, {\em Control of Turbulent and Magnetohydrodynamic Channel Flow}, Birkhauser, 2008.

\bibitem{vazquez2}
{\sc R.~Vazquez and M.~Krstic}, {\em Control of 1-D parabolic PDEs with Volterra nonlinearities --- Part I: Design},  Automatica, 44 (2008), pp.~2778--2790.

\bibitem{vazquez-obs}
{\sc R.~Vazquez, M.~Krstic and J.-M.~Coron}, {\em Backstepping Boundary Stabilization and State Estimation of a ${2\times2}$ Linear Hyperbolic System},  proceedings of the 50th IEEE CDC and ECC, 2011.

\bibitem{vazquez-of}
{\sc R.~Vazquez, M.~Krstic, J.-M.~Coron and G.~Bastin}, {\em Collocated Output-Feedback Stabilization of a 2 X 2 Quasilinear Hyperbolic System using Backstepping},  to appear in proceedings of the 2012 ACC.


\bibitem{xu}
{\sc C.~Z. Xu and G.~Sallet}, {\em Exponential stability and transfer functions of processes governed by symmetric hyperbolic systems}, ESAIM Control Optimisation and Calculus of Variations, 7 (2002), pp.~421--442.


\end{thebibliography}
\end{document}